\documentclass[11pt]{amsart}
\usepackage{mathrsfs}
\usepackage{amssymb}
\usepackage{graphicx}
\usepackage[T1]{fontenc}
\usepackage[all]{xy}
\usepackage{amscd}
\usepackage{amsmath, amssymb}
\usepackage{amsthm}
\usepackage{amsfonts}
\usepackage[colorlinks,linkcolor=blue,citecolor=blue, pdfstartview=FitH]
{hyperref}
\usepackage{backref}
\usepackage{amscd}
\usepackage{easybmat}
\usepackage{mathrsfs}
\usepackage{amsfonts}
\usepackage{color}
\usepackage{pifont}
\usepackage{upgreek}
\usepackage{bm}
\usepackage{hyperref}
\usepackage{shorttoc}
\usepackage{amsmath,amstext,amsthm,a4,amssymb,amscd}
\usepackage[mathscr]{eucal}
\usepackage{mathrsfs}
\usepackage{epsf}
\usepackage{tikz}

\numberwithin{equation}{section}

\newcommand{\dbar}{\ensuremath{\overline\partial}}

\newcommand{\B}{\ensuremath{\mathcal{B}}}
\newcommand{\X}{\ensuremath{\mathcal{X}}}

\def\k{\kappa}

\def\p{\partial}
\def\o{\overline}
\def\b{\bar}
\def\mb{\mathbb}
\def\mc{\mathcal}
\def\n{\nabla}

\newtheorem{thm}{Theorem}[section]
\newtheorem{lemma}[thm]{Lemma}
\newtheorem{prop}[thm]{Proposition}
\newtheorem{cor}[thm]{Corollary}
\theoremstyle{definition}
\newtheorem{rem}[thm]{Remark}

\theoremstyle{definition}
\newtheorem{defn}[thm]{Definition}

\theoremstyle{plain}
\newcommand{\comment}[1]{}

\usepackage{amscd}
\usepackage{fancyhdr}
\newenvironment{aligns}{\equation\aligned}{\endaligned\endequation}
\makeatletter
\@namedef{subjclassname@2020}{2020 Mathematics Subject Classification}

\begin{document}

\title{Curvature of the base manifold of a Monge-Amp\`ere fibration and its existence}

\author{Xueyuan Wan}
\author{Xu Wang}

\address{Xueyuan Wan: Mathematical Science Research Center, Chongqing University of Technology, Chongqing 400054, China}
\email{xwan@cqut.edu.cn}

\address{Xu Wang: Department of Mathematical Sciences, Norwegian University of Science and Technology, No-7491 Trondheim, Norway.}
\email{xu.wang@ntnu.no}

\begin{abstract}
In this paper, we consider a special relative K\"ahler fibration that satisfies a homogenous Monge-Amp\`ere equation, which is called a Monge-Amp\`ere fibration. There exist two canonical types of generalized Weil-Petersson metrics on the base complex manifold of the fibration. For the second generalized Weil-Petersson metric, we obtain an explicit curvature formula and prove that the holomorphic bisectional curvature is non-positive, the holomorphic sectional curvature, the Ricci curvature, and the scalar curvature are all bounded from above by a negative constant. For a holomorphic vector bundle over a compact K\"ahler manifold,  we prove that it admits a projectively flat Hermitian structure if and only if the associated projective bundle fibration is a Monge-Amp\`ere fibration. 
	In general, we can prove that a relative K\"ahler fibration is Monge-Amp\`ere  if and only if an associated infinite rank Higgs bundle is Higgs-flat. We also discuss some typical examples of Monge-Amp\`ere fibrations.  	
\end{abstract}

  \thanks{Xueyuan Wan is partially supported by the National Natural Science Foundation of China (grant no. 12101093) and Scientific Research Foundation of the Chongqing University of Technology. }

\maketitle
\tableofcontents

\section*{Introduction}

The curvature property of the moduli space of a holomorphic family of compact complex manifolds is an important research topic in complex geometry. 
For the moduli space of curves, there exists a classical Weil-Petersson metric, which is K\"ahler  \cite[Theorem 4]{A61}, and the  Ricci curvature, the
holomorphic sectional curvature and the scalar curvature are negative \cite[\S  10, Theorem]{A61b}, the holomorphic bisectional curvature is also negative \cite[Theorem 1.3]{LSYY}. There are also other curvature properties for the Weil-Petersson metric, such as negative sectional curvature \cite[Theorem 5]{Tromba}  \cite[Theorem 4.5]{ Wolpert}, strongly-negative curvature in the sense of Siu \cite[Theorem 1]{Sch0}, dual Nakano negative \cite[Theorem 4.1]{LSY}, non-positive Riemannian sectional curvature operator \cite[Theorem 1.1]{Wu},  etc. One can refer to \cite{LSYY} for the relations among these curvature properties of the Weil-Petersson metric. Moreover, by deriving an explicit formula for the curvature of the Weil-Petersson metric, S. Wolpert proved that the holomorphic sectional curvature,  the Ricci curvature, and the scalar curvature are all bounded above by a negative constant \cite[Lemma 4.6]{ Wolpert}. 

For the moduli space of compact K\"ahler-Einstein manifolds, there is a canonical metric, i.e.,  the generalized Weil-Petersson metric, which can be proved to be K\"ahler \cite[Theorem 12.3]{Koiso}. For the case of negative first Chern class, Y.-T. Siu \cite{Siu} computed the curvature of the generalized Weil-Petersson metric and obtained a  criterion on the negativity of the holomorphic bisectional curvature of the metric  \cite[Theorem 5.5]{Siu}. In \cite{Sch}, G. Schumacher considered the case of K\"ahler-Einstein manifolds with nonzero Ricci curvature $k$ and also gave an explicit formula \cite[Theorem 1]{Sch}. As an application,  for $k>0$,  he proved that the holomorphic sectional curvature and Ricci curvature of the generalized Weil-Petersson metric are bounded from below by a negative constant \cite[Corollary 1]{Sch}. For the moduli space of Calabi-Yau manifolds, G. Schumacher \cite{Sch-1} and G. Tian \cite{Tian}  showed that the generalized Weil-Petersson metric is K\"ahler.
 A. Nannicini \cite[proof of Theorem 1]{Nan} and A. N. Todorov \cite{Todorov} computed the curvature tensor of the generalized Weil-Petersson metric (two simple proofs of the curvature formula were given by C.-L. Wang \cite[Theorem 2.1]{WangCL} who also showed that both the holomorphic bisectional curvature and the Ricci curvature are bounded from below by a negative constant). In \cite{LS}, Z. Lu and X. Sun obtained an explicit formula for the curvature of partial Hodge metric  \cite[Theorem 1.1]{LS}. In the case of the moduli space of Calabi-Yau fourfolds, they proved that
the holomorphic bisectional curvature of the partial metric with a special factor (which is precisely the Hodge metric (up to a constant)) is non-positive, the Ricci curvature and the holomorphic sectional curvature are all bounded above by a negative constant \cite[Theorem 1.2]{LS}. For the general case, Z. Lu constructed a Hodge metric and proved that its holomorphic bisectional curvature is non-positive, the Ricci curvature and holomorphic sectional curvature are negative away from zero by a constant number \cite[Theorem 5.1]{Lu}. 
 For other related results, one can refer to \cite{BPW, Nau, Sch12}, etc. 

In this paper, we will study the curvature properties of the base complex manifold of a Monge-Amp\`ere fibration\footnote{The name of the Monge-Amp\`ere fibration was firstly given by Professor Bo Berndtsson.}, see Definition \ref{Monge-Ampere}. In \cite{Burns82}, D. Burns considered the curvature of a Monge-Amp\`ere foliation with only one-dimensional leaves (a local version of a Monge-Amp\`ere fibration) and obtained that the curvature is bounded from above by a negative constant \cite[Theorem 3.1]{Burns82}. A related negative curvature property for the space of all compatible almost complex structures was proven by Smolentsev in  \cite{Smo}. Let $p: (\mathcal X, \omega)\to\mathcal B$ be a relative K\"ahler fibration. It is called a Monge-Amp\`ere fibration if $\omega^{n+1}=0$, where $n$ denotes the dimension of each fiber. If the Kodaira-Spencer map is injective, then one can define two kinds of generalized Weil-Petersson metrics on the base complex manifold $\B$, i.e., $\omega_{\rm{WP}}$ and $\omega_{\mc{WP}}$, see Section \ref{WP metrics} for their definitions. The generalized Weil-Petersson metric $\omega_{\mc{WP}}$ is defined by the $\omega$-Kodaira-Spencer tensor $\kappa_j$ without taking harmonic projection,  so we always have $\omega_{\mc{WP}} \geq \omega_{\rm{WP}}$.   Our main result is the following curvature formula for the generalized Weil-Petersson metric $\omega_{\mc{WP}}$.

\begin{thm}\label{main theorem}
Let  $p:(\mc{X},\omega)\to \B$ be a Monge-Amp\`ere fibration with injective Kodaira-Spencer map. 
Then the metric $\omega_{\mc{WP}}$ is K\"ahler and its curvature is given by
	\begin{align*}
R_{j\b{k}l\b{m}}
=-\langle\o{\k_{m}}\k_j, \o{\k_{l}}\k_k\rangle-\langle\k_j\o{\k_{m}}, \k_k\o{\k_l}\rangle-\langle\operatorname{H}^{\perp}(L_{V_l}\k_j),  \operatorname{H}^{\perp}(L_{V_m}\k_{k})\rangle,
\end{align*}
where $\kappa_j$ is the $\omega$-Kodaira-Spencer tensor, see Definition \ref{de:KST}; $V_l$ denotes the horizontal lift of $\frac{\p}{\p t^l}$, see \eqref{eq:Vj}; the operator $L$ denotes the Lie derivative, $\operatorname{H}^{\perp}$ denotes the orthogonal projection from $A^{0,1}(X_t, T_{X_t})$ to  $\operatorname{Span}\{\k_{i}\}^{\perp}$.
\end{thm}
By using the above curvature formula, one can obtain some immediate consequences on various negativity results of different types of curvature.
\begin{cor}\label{main corollary}
	Let  $p:(\mc{X},\omega)\to \B$ be a Monge-Amp\`ere fibration with injective Kodaira-Spencer map. The holomorphic bisectional curvature of the generalized Weil-Petersson metric $\omega_{\mc{WP}}$ satisfies 
	\begin{align*}
	R(\xi,\o{\xi},\eta,\o{\eta})\leq -\frac{2}{n}|X_t|^{-1}|\langle\eta,\xi\rangle_{\mc{WP}}|^2
\end{align*}
for any two vectors $\eta,\xi$ in $T_t\mc{B}$, where $|X_t|:=\int_{X_t}\frac{\omega^n_t}{n!}$ denotes the volume of each fiber. In particular, we have the following negativity results of curvature\footnote{The negativity results of curvature are also obtained by Professor Bo Berndtsson independently using a different method based on the holomorphic motion structure of the fibration (see \cite{Bo-new}).}:
\begin{itemize}
\item[(i)] The holomorphic bisectional curvature is non-positive, and is negative if $\langle\eta,\xi\rangle_{\mc{WP}}\neq 0$;
\item[(ii)] The holomorphic sectional curvature and the Ricci curvature are both bounded from above by $-\frac{2}{n}|X_t|^{-1}$, the scalar curvature is bounded from above by $-\frac{2}{n}|X_t|^{-1}\dim\mc{B}$.
\end{itemize}
\end{cor}
Naturally, one may wonder what kind of relative K\"ahler fibration becomes a Monge-Amp\`ere fibration. In particular, for a holomorphic vector bundle over a compact complex manifold $\B$, there is a canonical relative K\"ahler fibration $p:P(E)\to \B$ with each fiber a projective space, where $P(E):=(E-\{0\})/\mb{C}^*$ denotes the projectivization of $E$. A natural question is for which holomorphic vector bundles $E$, the associated projective bundle fibration $p:P(E)\to \B$ is a Monge-Amp\`ere fibration. For this question, we have:

\begin{thm}\label{Theorem 0.3} Let $E$ be a holomorphic vector bundle over a compact K\"ahler manifold $\B$. Then  the following statements are equivalent:
\begin{itemize}
\item[1)]$E$ admits a projectively flat Hermitian structure;
\item[2)] $p:P(E)\to \B$ is a Monge-Amp\`ere fibration.
\end{itemize}
For the case of $\dim \B=1$, both are equivalent to the polystability of $E$.
\end{thm}
In Section \ref{Higgs bundle}, we shall introduce a finite rank Higgs bundle structure associated with a (non-proper) Monge-Amp\`ere fibration over the space $\mc{J}(V,\omega)$ of all $\omega$-compatible complex structures on a symplectic vector space $(V, \omega)$.
 This construction also suggests to introduce a certain infinite rank Higgs bundle for a general relative K\"ahler fibration. Let $\mc{A}:=\{\mc{A}_t\}_{t\in\B}$ be the space of smooth differential forms on $X_t$.  Denote by $\Gamma$ the space of all smooth sections of $\mc{A}$, see \eqref{Gamma}. With respect to the relative K\"ahler form $\omega$, there exists a Lie derivative connection $\n$ on $(\mc{A},\Gamma)$, see \eqref{Lie derivative connection},  which induces a Chern connection $D$ on $(\mc{A},\Gamma)$, see \eqref{Chern connection}, such that $\n-D=\theta+\bar{\theta}$ for a Higgs field $\theta$, where $\theta:= \sum dt^j \otimes \kappa_j$. Denoting by $(\mc{A},\Gamma, D,\theta)$ the associated Higgs bundle, we have:

\begin{thm}\label{Theorem 0.4}
A relative K\"ahler fibration $p: (\X,\omega) \to \B$ is a Monge-Amp\`ere fibration if and only if the following associated infinite rank Higgs bundle 
$$
(\mc{A},\Gamma, D,\theta)
$$
is Higgs-flat (cf. Proposition \ref{pr: flat-higgs}), where each fiber $\mathcal A_t$ denotes the space of smooth differential forms on $X_t$.	
\end{thm}
We also discuss some typical examples of Monge-Amp\`ere fibrations, which are also the motivations for studying such kind of relative K\"ahler fibration. For example, the family of elliptic curves, finite rank Higgs bundle version of a (non-proper) Monge-Amp\`ere fibration, and various kinds of geodesics.

This article is organized as follows. In Section \ref{sec1}, we review some basic definitions and facts on the relative K\"ahler fibrations, Monge-Amp\`ere fibrations, and two types of generalized Weil-Petersson metrics. In Section \ref{sec2}, we will compute the curvature of the generalized Weil-Petersson metric $\omega_{\mc{WP}}$, and we will prove Theorem \ref{main theorem} and Corollary \ref{main corollary}. In Section \ref{Sec3}, we will consider the existence of Monge-Amp\`ere fibrations. In Section \ref{sec3.1}, we will show that a holomorphic vector bundle admits a projectively flat Hermitian structure if and only if the associated projective bundle fibration is a Monge-Amp\`ere fibration, and prove Theorem \ref{Theorem 0.3}. In Section \ref{sec3.2}, we will prove that a relative K\"ahler fibration is Monge-Amp\`ere if and only if an associated infinite rank Higgs bundle is Higgs-flat, and prove Theorem \ref{Theorem 0.4}. The last section will give some typical examples of Monge-Amp\`ere fibrations.

\vspace{5mm}
\noindent{\it Acknowledgements.} We would like to thank Bo Berndtsson and Ya Deng for several useful discussions about the topics of this paper. We also would like to thank the anonymous reviewers for their comments that helped improve the paper.

\section{Preliminaries}\label{sec1}

In this section, we will review some basic definitions and facts on the relative K\"ahler fibrations, Monge-Amp\`ere fibrations, and two types of generalized Weil-Petersson metrics.

\subsection{Relative K\"ahler fibrations} 
Let $\mc{X}$ and $\mc{B}$ be two complex manifolds. 
\begin{defn}\label{de:rkf} We call a proper holomorphic submersion $p: (\mathcal X, \omega)\to\mathcal B$ between two complex manifolds {\it a relative K\"ahler fibration} if $\omega$ is a real, smooth, $d$-closed $(1,1)$-form on $\X$ and $\omega$ is positive on each fiber $X_t:=p^{-1}(t)$ of $p$.
\end{defn}
\begin{defn}\label{de:horizontal-vector-field} Let $p: (\mathcal X, \omega)\to\mathcal B$ be a relative K\"ahler fibration. By vertical vector fields, we mean vector fields on $\mathcal X$ that are tangent to the fibers,  a vector field $V$ on $\mathcal X$ is said to be {\it horizontal} with respect to $\omega$ if 
$$
\omega(V, W)=0
$$
for every vertical $W$.
\end{defn}

The relative K\"ahler form $\omega$ defines a natural inner product (not semi-positive in general) such that
\begin{equation}\label{eq:hermitian}
\langle V, W\rangle_{\omega}=\omega (V, J\overline W),
\end{equation}
where $J$ denotes the complex structure on $\mathcal X$. We say that $V$ is \emph{orthogonal} to $W$ with respect to $\omega$ if $
\langle V, W\rangle_{\omega}=0$.
Thus a vector field is horizontal if and only if it is orthogonal to all vertical vector fields.

\begin{defn}\label{de:horizontal-lift} Let $p: (\mathcal X, \omega)\to\mathcal B$ be a relative K\"ahler fibration, and let $v$ be a vector field on $\mathcal B$. A vector field $V$ on $\mathcal X$ is said to be {\it a horizontal lift} of $v$ with respect to $\omega$ if $V$ is horizontal and $p_*(V)=v$.
\end{defn}

For the horizontal lift of a vector field, we have the following proposition (see e.g. \cite[Section 4.1]{BPW}).
\begin{prop} Every vector field on $\mathcal B$ has a unique horizontal lift. Horizontal lift of a $(1,0)$-vector field (resp. $(0,1)$-vector field) is still a $(1,0)$-vector field (resp. $(0,1)$-vector field).
\end{prop}

 Let $\{t^j\}$ be a holomorphic local coordinate system on $\mathcal B$. Since $p$ is a holomorphic fibration, we can find ${\zeta^\alpha}$ such that $\{t^j, \zeta^\alpha\}$ is a holomorphic local coordinate system on $\mathcal X$. Since $\omega$ is a closed $(1,1)$ form, we write it locally as $\omega=i\partial\dbar \phi$ for some local real function $\phi$.  Then we know that each
\begin{equation}\label{eq:Vj}
V_j:=\frac{\partial}{\partial t^j}-\sum_{\beta=1}^n \phi_{j \bar\beta} \phi^{\bar \beta\alpha } \frac{\partial}{\partial \zeta^\alpha},   \  \ \phi_{j\bar\beta}:=\frac{\partial^2 \phi}{\partial t^j \partial \bar\zeta^\beta},
\end{equation}
is a horizontal lift of $\frac{\partial}{\partial t^j}$, where $(\phi^{\bar \beta\alpha})$ denotes the inverse matrix of $(\phi_{\alpha\bar\beta})$ and $\phi_{\alpha\bar\beta}:=\frac{\partial^2 \phi}{\partial \zeta^\alpha \partial \bar\zeta^\beta}$,  $n$ denotes the complex dimension of each fiber. Denote 
\begin{equation}\label{eq: cjk}
c_{j\bar k}:=\langle V_j, V_k\rangle_{\omega}=\phi_{j\b{k}}-\sum_{\alpha,\beta=1}^n\phi_{j\b{\beta}}\phi^{\b{\beta}\alpha}\phi_{\alpha\b{k}}, \
c(\omega):=i \sum_{j,k=1}^{\dim\mathcal{B}}  c_{j\bar k}\, dt^j\wedge d\bar t^k.
\end{equation}
 We call $c_{j\bar k}$ the {\it geodesic curvatures} and $c(\omega)$ the {\it geodesic curvature form}.
A direct calculation shows that 
\begin{align}\label{decomposition}
\begin{split}
  \omega=i\p\b{\p}\phi=c(\omega)+\omega_{\mc{X}/\B},\quad \omega_{\mc{X}/\B}:=i\sum_{\alpha,\beta=1}^n\phi_{\alpha\b{\beta}}\delta \zeta^\alpha\wedge \delta \b{\zeta}^\beta,
 \end{split}
\end{align}
where $\delta \zeta^\alpha=d\zeta^\alpha+\sum_{\beta,j} \phi_{j\b{\beta}}\phi^{\b{\beta}\alpha} d t^j$.
The following proposition is a generalization of \cite[Lemma 6.1]{Wang15}.
\begin{prop}\label{pr:rkf} Let $\{V_j\}$ be the vector fields defined in \eqref{eq:Vj}, $\dim X_t=n$. Then 
\begin{enumerate}
\item $[V_j, V_k]=0$;
\item $(\omega-c(\omega))^{n+1}=0$;
\item $[V_j, \bar V_k] \,\rfloor\, (\omega|_{X_t})=i (dc_{j\bar k})|_{X_t}$;
\item $[V_j, \bar V_k]\equiv 0$ for all $j,k$ if and only if $d(c(\omega))=0$.
\end{enumerate}
\end{prop}
\begin{proof}
\begin{itemize}
  \item[(1)] By a direct computation, we know that $[V_j,  V_k] $ are vertical. Since $\omega$ is non-degenerate on fibers, it is enough to prove that $[V_j, V_k] \, \rfloor \, \omega=0$ on fibers. Notice that
$$
[V_j, V_k] \, \rfloor \, \omega=(L_{V_j}V_k) \, \rfloor \, \omega=L_{V_j}(V_k\, \rfloor \, \omega)-V_k \, \rfloor \, L_{V_j}\omega,
$$
and by \eqref{eq:Vj} we have
\begin{equation}\label{eq:vj-omega}
V_j\, \rfloor \, \omega=i \sum_{l=1}^{\dim\B} c_{j\bar l} \, d\bar t^l.
\end{equation}
By using the Cartan formula, we get
\begin{equation}\label{eq:wang-1}
[V_j, V_k] \, \rfloor \, \omega=i \sum_{l=1}^{\dim\B} (V_j \, \rfloor\, d c_{k\bar l} )\, d\bar t^l- i \sum_{l=1}^{\dim\B} (V_k \, \rfloor\, d c_{j\bar l} )\, d\bar t^l.
\end{equation}
Thus $[V_j, V_k] \, \rfloor \, \omega=0$ on fibers, and so $[V_j,V_k]=0$.
\item[(2)] From \eqref{decomposition}, one has 
\begin{align*}
\begin{split}
  (\omega-c(\omega))^{n+1}=\omega_{\mc{X}/\B}^{n+1}=\left(i\sum_{\alpha,\beta=1}^n\phi_{\alpha\b{\beta}}\delta \zeta^\alpha\wedge \delta \b{\zeta}^\beta\right)^{n+1}=0.
 \end{split}
\end{align*}

\item[(3)] Notice that 
$$
[V_j, \bar V_k] \, \rfloor \, \omega=(L_{V_j}\bar V_k) \, \rfloor \, \omega=L_{V_j}(\bar V_k\, \rfloor \, \omega)-\bar V_k \, \rfloor \, L_{V_j}\omega
$$
and combining with \eqref{eq:vj-omega} we have
$$
[V_j, \bar V_k] \, \rfloor \, \omega= i \ d c_{j\bar k}-i \sum_{l=1}^{\dim\B} (V_j \, \rfloor\, d c_{l \bar k} )\, d t^l- i \sum_{l=1}^{\dim\B} (\bar V_k \, \rfloor\, d c_{j\bar l} )\, d\bar t^l.
$$
Since $[V_j,  \bar V_k]$ is vertical, so
\begin{align*}
\begin{split}
 [V_j, \bar V_k] \,\rfloor\, (\omega|_{X_t})=\left([V_j, \bar V_k] \, \rfloor \, \omega\right)|_{X_t}=i (dc_{j\bar k})|_{X_t},
 \end{split}
\end{align*}
which proves $(3)$.

\item[(4)] By $(3)$, we know that $dc(\omega)=0$ gives  $[V_j, \bar V_k]\equiv 0$. For the opposite direction, assume that  $[V_j, \bar V_k]\equiv 0$ all for $j,k$, then by $(3)$, we know that $c_{j\bar k}$ depends only on $t\in \B$, thus by $(1)$ and \eqref{eq:wang-1}, we have
$$
0=[V_j, V_k] \, \rfloor \, \omega=i \sum_{l=1}^{\dim\B} \frac{\partial c_{k\bar l}}{\partial t^j} \, d\bar t^l- i \sum_{l=1}^{\dim\B}\frac{\partial c_{j\bar l}}{\partial t^k} \, d\bar t^l,
$$
which implies that $c(\omega)$ is $d$-closed. Thus $dc(\omega)=0$.
\end{itemize}
\end{proof}
\begin{rem}
From (1) and (4) in Proposition \ref{pr:rkf},
the horizontal distribution of a relative K\"ahler fibration is integrable if and only if each geodesic curvature $c_{j\bar k}$ is constant on fibers, which determines a differentiable trivialization of the fibration.
\end{rem}
\begin{rem}
If the geodesic curvature form $c(\omega)$ depends only on the  base $\B$, then 
\begin{align*}
\begin{split}
  \int_{\mc{X}/\B}\frac{\omega^{n+1}}{(n+1)!}=\int_{\mc{X}/\B}c(\omega)\wedge \frac{\omega_{\mc{X}/\B}^n}{n!}=c(\omega)\int_{X_t}\frac{(\omega|_{X_t})^n}{n!}=c(\omega)|X_t|,
 \end{split}
\end{align*}
	where $|X_t|:=\int_{X_t}\frac{(\omega|_{X_t})^n}{n!}$ denotes the volume of each fiber. Hence
	\begin{align*}
\begin{split}
  c(\omega)=\frac{1}{|X_t|} \int_{\mc{X}/\B}\frac{\omega^{n+1}}{(n+1)!},
 \end{split}
\end{align*}
which is $d$-closed. 
\end{rem}

\subsection{Monge-Amp\`ere fibrations}

In this subsection, we will give the definition of a Monge-Amp\`ere fibration.
\begin{defn}\label{Monge-Ampere} A relative K\"ahler fibration $p: (\mathcal X, \omega)\to\mathcal B$ is said to be {\it Monge-Amp\`ere} (we say that $\omega$ is {\it a Monge-Amp\`ere form})  if $\omega$ solves the homogeneous complex Monge--Amp\`{e}re equation, i.e.
$$
\omega^{n+1}\equiv 0,
$$
where $n$ denotes the dimension of the fibers. In general,  a proper holomorphic submersion $p: (\X, \omega_\X) \to (\B, \omega_\B)$ between two K\"ahler manifolds is said to be {\it Monge-Amp\`ere} if 
$$
(\omega_\X- p^* \omega_\B)^{n+1}\equiv 0,
$$
(in which case we know  $\omega_\X- p^* \omega_\B$ is {\it a Monge-Amp\`ere form}). A proper holomorphic submersion $p:\mc{X}\to \B$ is called a Monge-Amp\`ere fibration if there exists a Monge-Amp\`ere form $\omega$ on $\mc{X}$.
\end{defn}
\begin{rem}
	\begin{itemize}
  \item[(1)] By Proposition \ref{pr:rkf}, for a relative K\"ahler fibration $p: (\mathcal X, \omega)\to\mathcal B$, $d\omega'=0$ if and only if $[V_j, \bar V_k]\equiv 0$ all for $j,k$, where $\omega'=\omega-c(\omega)$.  Thus $\omega'$ is a Monge-Amp\`ere form if and only if the horizontal distribution associated with $\omega$ is integrable. 
  \item[(2)] A relative K\"ahler fibration $p: (\mathcal X, \omega)\to\mathcal B$ is a Monge-Amp\`ere fibration if and only if 
  \begin{align*}
\begin{split}
  0&=\omega^{n+1}=(c(\omega)+\omega_{\mc{X}/\mc{B}})^{n+1}\\
  &=(n+1)\omega_{\mc{X}/\mc{B}}^{n}\wedge c(\omega)+\sum_{i=2}^{n+1}C_{n+1}^i\omega_{\mc{X}/\mc{B}}^{n+1-i}\wedge c(\omega)^i,
 \end{split}
\end{align*}
which is equivalent to $c(\omega)\equiv 0$.
\item[(3)] If $p:(\mc{X},\omega)\to \B$ is a Monge-Amp\`ere form, then the $d$-closed $(1,1)$-form $\int_{\mc{X}/\B}\omega^{n+1}$ vanishes. 
\end{itemize}
\end{rem}

\subsection{Generalized Weil-Petersson metrics}\label{WP metrics}
In this subsection, by using the relative K\"ahler form $\omega$, we shall define two types of generalized Weil-Petersson metrics on the base manifold of a Monge-Amp\`ere fibration.

\begin{defn}\label{de:KST} Let $p: (\mathcal X,  \omega)\to \mathcal B$ be a relative K\"ahler fibration. Let $V_j$ (defined in \eqref{eq:Vj}) be the {horizontal lift} of $\frac{\partial}{\partial t^j}$ with respect to $\omega$. We call
$$
\kappa_j:=(\dbar V_j)|_{X_t}
$$
the {\it $\omega$-Kodaira--Spencer tensor} on $X_t$.
\end{defn}

 From the above definition, one sees that each $\omega$-Kodaira--Spencer tensor $\kappa_j$ is a $\dbar$-closed  $T_{X_t}$-valued $(0,1)$-form on $X_t$. By using the $\omega$-Kodaira--Spencer tensor, the generalized Weil-Petersson metric can be given as follows, see {\cite[Definition 7.1]{FS}}.
 \begin{defn}
 	Let $p: (\mathcal X,  \omega)\to \mathcal B$ be a  relative K\"ahler fibration. We call the following metric on $\mathcal B$ defined by
$$
\left\langle \frac{\partial}{\partial t^j},\frac{\partial}{\partial t^k} \right\rangle_{\rm WP}(t):=\int_{X_t} \langle \kappa_j^h , \kappa_k^h \rangle_{\omega_t} \,\frac{\omega_t^n}{n!}, \ \ \omega_t:= \omega|_{X_t},
$$
the {\it generalized Weil-Petersson metric} on $\mathcal B$, where $\kappa_j^h$ denotes the $\omega_t$ harmonic representative of the Kodaira--Spencer class $[\kappa_j]$.
 \end{defn}
On the other hand, one can take the $L^2$-inner product of the \emph{$\omega$-Kodaira--Spencer tensors} $\kappa_j$ directly (without taking the harmonic projection), which gives the following definition of generalized Weil-Petersson metrics, see \cite[Section 8]{FS}.  
\begin{defn}\label{de:DFWP} Let $p: (\mathcal X,  \omega)\to \mathcal B$ be a  relative K\"ahler fibration. We can define another kind of 
generalized Weil-Petersson metric on $\mathcal B$  by
$$
\left\langle \frac{\partial}{\partial t^j},\frac{\partial}{\partial t^k} \right\rangle_{\mc{W}\mc{P}}(t):=\int_{X_t} \langle \kappa_j, \kappa_k \rangle_{\omega_t} \,\frac{\omega_t^n}{n!}, \ \ \omega_t:= \omega|_{X_t},
$$
where $\kappa_j$ are $\omega$-Kodaira--Spencer tensors. 
\end{defn}

One may note that the generalized  Weil-Petersson  metric $\left\langle\cdot,\cdot\right\rangle_{\mc{WP}}$ is  \emph{bigger} than  $\left\langle\cdot,\cdot\right\rangle_{\rm WP}$. In particular, if the Kodaira-Spencer map is injective, then both kinds of generalized Weil-Petersson metrics must be non-degenerated.
\begin{rem}
	 It is proved in \cite{Wang16} that if the relative cotangent bundle is $(n-1)$-semi-positive, then the bisectional curvature of the \emph{generalized Weil-Petersson metric} is semi-negative. But in general, it is not easy to find such fibrations with  $(n-1)$-semi-positive relative cotangent bundle. The main theme of this paper is to use the \emph{generalized Weil-Petersson metric} $\left\langle\cdot,\cdot\right\rangle_{\mc{WP}}$ to study the curvature properties of the base manifold of a Monge-Amp\`ere fibration.

\end{rem}

\section{Curvature of the generalized Weil-Petersson metric}\label{sec2}

Let $p: (\mc{X},\omega)\to \mc{B}$ be a relative K\"ahler fibration, i.e., $\omega$ is a real and smooth $d$-closed $(1,1)$-form on $\mc{X}$, and is positive on each fiber $X_t:=p^{-1}(t)$.  By $\bar{\p}$-Poincar\'e Lemma, there exists a local weight, say $\phi$, such that 
\begin{align*}
\omega=i\p\b{\p}\phi. 	
\end{align*}
Let $\{t^j, \zeta^{\alpha}\}$ denote a holomorphic local coordinate system on $\mc{X}$ such that $p(t,\zeta)=t$. Then 
\begin{align*}
\omega=i\left(\phi_{\alpha\b{\beta}}d\zeta^{\alpha}\wedge d\b{\zeta}^{\beta}+\phi_{i\b{\beta}}dt^j\wedge d\b{\zeta}^{\beta}+\phi_{\alpha\b{k}}d\zeta^{\alpha}\wedge d\b{t}^k+\phi_{j\b{k}}dt^j\wedge d\b{t}^k\right),
\end{align*}
where $\phi_{A\b{B}}:=\p_A\p_{\b{B}}\phi$. In this section, we will use the summation convention of Einstein. Recall the canonical horizontal lift of $\frac{\p}{\p t^j}$ is given by 
\begin{align*}
V_j:=\frac{\p}{\p t^j}-\phi_{j\b{\beta}}\phi^{\b{\beta}\alpha}\frac{\p}{\p\zeta^{\alpha}},
\end{align*}
and recall the $\omega$-Kodaira-Spencer tensor on $X_t$ is given by
\begin{align*}
\k_j:=(\b{\p}V_j)|_{X_t}.	
\end{align*}
The generalized Weil-Petersson metric  $\left\langle\cdot,\cdot\right\rangle_{\mc{WP}}$ is then defined by 
\begin{align*}
\left\langle \frac{\p}{\p t^j},\frac{\p}{\p t^k}\right\rangle_{\mc{WP}}(t):=\int_{X_t}\langle \k_j,\k_k\rangle_{\omega_t}\frac{\omega^n_t}{n!}, \quad \omega_t=\omega|_{X_t}.	
\end{align*}
Denote 
\begin{align*}
\omega_{\mc{WP}}=iG_{j\b{k}}dt^j\wedge d\b{t}^k,\quad G_{j\b{k}}:=\left\langle \frac{\p}{\p t^j},\frac{\p}{\p t^k}\right\rangle_{\mc{WP}}.
\end{align*}
With respect to $\omega$, recall  that the geodesic curvature form is given by
\begin{align*}
c(\omega)=ic_{j\b{k}}dt^j\wedge d\b{t}^k, \quad c_{j\b{k}}:=\langle V_j, V_k\rangle_{\omega}=\phi_{j\b{k}}-\phi_{j\b{\beta}}\phi^{\alpha\b{\beta}}\phi_{\alpha\b{k}}.
\end{align*}
If each fiber $X_t$ is compact, Fujiki and Schumacher \cite{FS} obtained the following expression on the generalized Weil-Petersson metric $\omega_{\mc{WP}}$, see also \cite[Lemma 3.8 (3.43)]{Wan1} for its proof.
\begin{thm} [{\cite[Theorem 8.1]{FS}}]\label{Expression of WP}The following identity holds
\begin{align}\label{le1}
\omega_{\mc{WP}}=i\int_{\mc{X}/\mc{B}}R^{K_{\mc{X}/\mc{B}}}\wedge \frac{\omega^n}{n!}+\int_{\mc{X}/\mc{B}}\rho c(\omega)\wedge\frac{\omega^n}{n!},
\end{align}
where $R^{K_{\mc{X}/\mc{B}}}=\p\b{\p}\log\det \phi$, $\rho=-\phi^{\alpha\b{\beta}}\p_\alpha\p_{\b{\beta}}\log\det \phi$ is the saclar curvature, $\det \phi:=\det(\phi_{\alpha\b{\beta}})$, $\int_{\mc{X}/\mc{B}}$ denotes fiber integration (see e.g. \cite[Section 2.1]{Sch12} for fiber integration).
	\end{thm} 
 As a corollary, one has
 \begin{cor}
 	If $p:(\mc{X},\omega)\to \mc{B}$ is a Monge-Amp\`ere fibration, then 
 	\begin{align}\label{le2}
\omega_{\mc{WP}}=i\int_{\mc{X}/\mc{B}}R^{K_{\mc{X}/\mc{B}}}\wedge \frac{\omega^n}{n!}.
\end{align}
In particular, $\omega_{\mc{WP}}$ is $d$-closed. 
 \end{cor}

Now we will follow Schumacher's method \cite{Sch} to calculate the curvature of generalized Weil-Petersson metric $\omega_{\mc{WP}}$. 
Let  $T_{X_t}$ denote the holomorphic tangent bundle of $X_t$, and denote by $T_{X_t}^{\mb{C}}=T_{X_t}\oplus\o{T_{X_t}}$ the complexified tangent bundle.
For any two tensors
$$\Phi=\Phi^{A}_Bdx^B\otimes \frac{\p}{\p x^A},\quad \Psi=\Psi^A_Bdx^B\otimes\frac{\p}{\p x^A}\in A^{1}(X_t, T_{X_t}^{\mb{C}})\simeq A^0(X_t, \text{End}(T^{\mb{C}}_{X_t})),$$ where $x^A, x^B$ are taken $\{\zeta^{\alpha}, \b{\zeta}^{\beta}\}$. We define 
\begin{align*}
\Phi\cdot\Psi:=\text{Tr}(\Phi\Psi)=\Phi^A_B\Psi^B_A.
\end{align*}
For any vector field $V$, we denote by $L_V$ the Lie derivative along $V$.  For the tensor $\Phi=\Phi^{A}_Bdx^B\otimes \frac{\p}{\p x^A}\in A^{1}(X_t, T_{X_t}^{\mb{C}})$, one has 
\begin{align}\label{L1}
L_V\Phi=\left(L_V\Phi^{A}_{B}\right)\frac{\p}{\p x^{A}}\otimes dx^{B},
\end{align}
where  
\begin{align}\label{L2}
\begin{split}
	L_V\Phi^{A}_{B}&=V(\Phi^{A}_{B})-\Phi^C_B\frac{\p V^{ A}}{\p x^C}+\Phi^A_C\frac{\p V^{C}}{\p x^{B}}\\
	&=\n_V(\Phi^{A}_{B})-\Phi^{C}_{B}\n_C V^{A}+\Phi^{A}_{C}\n_{B} V^{C}.
	\end{split}
\end{align}
Here $\n_C$ denotes the covariant derivative along $\p/\p x^{C}$ with respect to some Hermitian metric. 
Since Lie derivative commutes with contraction and satisfies  Leibniz's rule for tensors, so 
\begin{align*}
L_V	(\Phi\cdot\Psi)=(L_V\Phi)\cdot\Psi+\Phi\cdot(L_V\Psi). 
\end{align*}
Denote 
\begin{align*}
\k_j=A^{\alpha}_{j\b{\beta}}d\b{\zeta}^{\beta}\otimes \frac{\p}{\p \zeta^{\alpha}},\quad A^{\alpha}_{j\b{\beta}}=-\p_{\b{\beta}}(\phi_{j\b{\gamma}}\phi^{\b{\gamma}\alpha}).
\end{align*}
By a direct calculation, one has 
\begin{align}\label{2.1}
	A^{\alpha}_{j\b{\beta}}=A^{\sigma}_{j\b{\gamma}}\phi^{\b{\gamma}\alpha}\phi_{\sigma\b{\beta}},
\end{align}
(see e.g. \cite[(3.12)]{Wan1}). Then
\begin{align*}
	\langle \k_j,\k_k\rangle_{\omega_t}=A^{\alpha}_{j\b{\beta}}\o{A^{\sigma}_{k\b{\gamma}}}\phi^{\gamma\b{\beta}}\phi_{\alpha\b{\sigma}}=A^{\alpha}_{j\b{\beta}}\o{A^\beta_{k\b{\alpha}}}=\k_j\cdot \o{\k_{k}}.
\end{align*}
The first variation of the generalized Weil-Petersson metric is 
\begin{aligns}\label{2.4}
\frac{\p G_{j\b{k}}}{\p t^l}&=\frac{\p}{\p t^l}\int_{X_t}\k_j\cdot\o{\k_{k}}\frac{\omega^n_t}{n!}	\\
&=\int_{X_t} (L_{V_l}\k_j)\cdot\o{\k_{k}}\frac{\omega^n_t}{n!}+\int_{X_t}\k_j\cdot L_{V_l}\o{\k_{k}}\frac{\omega^n_t}{n!}+\int_{X_t}\k_j\cdot\o{\k_{k}}L_{V_j}\frac{\omega^n_t}{n!}\\
&=\int_{X_t} (L_{V_l}\k_j)\cdot\o{\k_{k}}\frac{\omega^n_t}{n!}+\int_{X_t}\k_j\cdot L_{V_l}\o{\k_{k}}\frac{\omega^n_t}{n!},
\end{aligns}
where the second equality follows from \cite[Lemma 1]{Sch12}, the last equality holds by \cite[Lemma 2.2 (2)]{Sch}. From \cite[Lemma 2.3]{Sch} or  (\ref{L1}), (\ref{L2}), one has
\begin{aligns}\label{2.6}
L_{V_l}\o{\k_{k}}&=L_{V_l}(\o{A^{\beta}_{k\b{\alpha}}} d\zeta^{\alpha}\otimes \frac{\p}{\p\b{\zeta}^{\beta}} )\\
&=-(c_{l\b{k}})^{;\b{\beta}}_{~~\alpha}d\zeta^{\alpha}\otimes \frac{\p}{\p\b{\zeta}^{\beta}} -A^{\gamma}_{l\b{\beta}}\o{A^{\beta}_{k\b{\alpha}}}d\zeta^{\alpha}\otimes \frac{\p}{\p\zeta^{\gamma}}+\o{A^{\beta}_{k\b{\alpha}}}A^{\alpha}_{l\b{\delta}}d\b{\zeta}^{\delta}\otimes\frac{\p}{\p\b{\zeta}^{\beta}}\\
&=-(c_{l\b{k}})^{;\b{\beta}}_{~~\alpha}d\zeta^{\alpha}\otimes \frac{\p}{\p\b{\zeta}^{\beta}}-\k_{l}\o{\k_{k}}+\o{\k_{k}}\k_{l}.
\end{aligns}
Thus 
\begin{aligns}\label{2.2}
	\int_{X_t}\k_j\cdot L_{V_l}\o{\k_{k}}\frac{\omega^n_t}{n!} &=-\int_{X_t}A^{\alpha}_{j\b{\beta}} (c_{l\b{k}})^{;\b{\beta}}_{~~\alpha}\frac{\omega^n_t}{n!}\\
	&=-\int_{X_t} (A^{\alpha}_{j\b{\beta}})_{;\alpha}^{~~;\b{\beta}}c_{l\b{k}}\frac{\omega^n_t}{n!}=\int_{X_t} (V_j\rho) c_{l\b{k}}\frac{\omega^n_t}{n!},
\end{aligns}
where the last equality follows from 
\begin{align*}
	(A^{\alpha}_{j\b{\beta}})_{;\alpha}^{~~;\b{\beta}}&=(A^{\alpha}_{j\b{\beta}})_{;\alpha\gamma}\phi^{\gamma\b{\beta}}=-\phi_{j\b{\sigma};\b{\beta}\alpha\gamma}\phi^{\b{\sigma}\alpha}\phi^{\gamma\b{\beta}}\\
	&=-(\phi_{j\b{\sigma};\alpha\b{\beta}}+R_{\alpha\b{\sigma}\tau\b{\beta}}\phi^{\b{\tau}\b{\delta}}\phi_{j\b{\delta}})_{;\gamma}\phi^{\b{\sigma}\alpha}\phi^{\gamma\b{\beta}}\\
	&=-(\p_j\phi_{\alpha\b{\sigma}})_{;\b{\beta}\gamma}\phi^{\b{\sigma}\alpha}\phi^{\gamma\b{\beta}}-(R_{\alpha\b{\sigma}\tau\b{\beta}}\phi^{\b{\tau}\b{\delta}}\phi_{j\b{\delta}})_{;\gamma}\phi^{\b{\sigma}\alpha}\phi^{\gamma\b{\beta}}\\
	&=-\p_j\p_{\gamma}\p_{\b{\beta}}\log\det \phi \phi^{\gamma\b{\beta}}+(\p_{\tau}\p_{\b{\beta}}\log\det \phi \phi^{\tau\b{\delta}}\phi_{j\b{\delta}})_{;\gamma}\phi^{\gamma\b{\beta}}\\
	&=-\p_j\rho+\p_{\gamma}\p_{\b{\beta}}\log\det \phi \p_j\phi^{\gamma\b{\beta}}+\p_{\tau}\p_{\b{\beta}}\log\det \phi \phi^{\tau\b{\delta}}\phi_{j\gamma\b{\delta}}\phi^{\gamma\b{\beta}}\\
	&\quad+(\p_{\tau}\p_{\b{\beta}}\log\det \phi)_{;\gamma} \phi^{\tau\b{\delta}}\phi_{j\b{\delta}}\phi^{\gamma\b{\beta}}\\
	&=-\p_j\rho+(\p_{\gamma}\p_{\b{\beta}}\log\det \phi)_{;\tau} \phi^{\tau\b{\delta}}\phi_{j\b{\delta}}\phi^{\gamma\b{\beta}}=-V_j\rho.
\end{align*}
On the other hand, by (\ref{L1}) and (\ref{L2}), one has 
\begin{align}\label{2.7}
\begin{split}
	L_{V_l}\k_j&=(L_{V_l}\k_j)^{\alpha}_{\b{\beta}}d\b{\zeta}^{\beta}\otimes\frac{\p}{\p \zeta^{\alpha}}\\&=\left(\p_l(A^{\alpha}_{j\b{\beta}})-\phi_{l\b{\gamma}}\phi^{\b{\gamma}\sigma}A^{\alpha}_{j\b{\beta};\sigma}+A^{\sigma}_{j\b{\beta}}\phi_{l\sigma\b{\gamma}}\phi^{\b{\gamma}\alpha}\right)d\b{\zeta}^{\beta}\otimes\frac{\p}{\p \zeta^{\alpha}}.
	\end{split}
\end{align}
By a direct calculation, one has 
\begin{align}\label{2.13}
	(L_{V_l}\k_j)^{\alpha}_{\b{\beta}}=(L_{V_l}\k_j)^{\tau}_{\b{\delta}}\phi^{\b{\delta}\alpha}\phi_{\tau\b{\beta}}.
\end{align}
In fact, by (\ref{2.1}), one has
\begin{align*}
	(L_{V_l}\k_j)^{\alpha}_{\b{\beta}}&=\p_l(A^{\alpha}_{j\b{\beta}})-\phi_{l\b{\gamma}}\phi^{\b{\gamma}\sigma}A^{\alpha}_{j\b{\beta};\sigma}+A^{\sigma}_{j\b{\beta}}\phi_{l\sigma\b{\gamma}}\phi^{\b{\gamma}\alpha}\\
	&=\p_l(A^{\alpha}_{j\b{\beta}})-A^{\sigma}_{j\b{\gamma}}\phi_{\sigma\b{\beta}}\p_l\phi^{\b{\gamma}\alpha}-(\phi_{l\b{\gamma}}\phi^{\b{\gamma}\sigma}A^{\tau}_{j\b{\delta};\sigma})\phi^{\b{\delta}\alpha}\phi_{\tau\b{\beta}}\\
	&=\p_l A^{\sigma}_{j\b{\gamma}}\phi_{\sigma\b{\beta}}\phi^{\b{\gamma}\alpha}+ A^{\sigma}_{j\b{\gamma}}\p_l \phi_{\sigma\b{\beta}}\phi^{\b{\gamma}\alpha}-(\phi_{l\b{\gamma}}\phi^{\b{\gamma}\sigma}A^{\tau}_{j\b{\delta};\sigma})\phi^{\b{\delta}\alpha}\phi_{\tau\b{\beta}}\\
	&=(\p_l(A^{\tau}_{j\b{\delta}})-\phi_{l\b{\gamma}}\phi^{\b{\gamma}\sigma}A^{\tau}_{j\b{\delta};\sigma}+A^{\sigma}_{j\b{\delta}}\phi_{l\sigma\b{\gamma}}\phi^{\b{\gamma}\tau})\phi^{\b{\delta}\alpha}\phi_{\tau\b{\beta}}\\
	&=(L_{V_l}\k_j)^{\tau}_{\b{\delta}}\phi^{\b{\delta}\alpha}\phi_{\tau\b{\beta}},
\end{align*}
which completes the proof of (\ref{2.13}).
Combining with (\ref{2.1}), we have
\begin{align}\label{2.3}
\int_{X_t} (L_{V_l}\k_j)\cdot\o{\k_{k}}\frac{\omega^n_t}{n!}=\langle L_{V_l}\k_j, \k_{k}\rangle.
\end{align}
Here \begin{align*}\langle\cdot,\cdot\rangle:=\int_{X_t}\langle\cdot,\cdot\rangle_{\omega_t}\frac{\omega^n_{t}}{n!}\end{align*} denotes the  global $L^2$-inner product. Substituting (\ref{2.2}) and (\ref{2.3}) into (\ref{2.4}), we obtain
\begin{prop}\label{prop1}
	Let $p: (\mc{X},\omega)\to \mc{B}$ be a relative K\"ahler fibration with compact fibers.  The first variation of the generalized Weil-Petersson  metric is 
	\begin{align*}
	\frac{\p G_{j\b{k}}}{\p t^l}=	\langle L_{V_l}\k_j, \k_{k}\rangle+\int_{X_t} (V_j\rho) c_{l\b{k}}\frac{\omega^n_t}{n!}.
	\end{align*}
In particular, if $\rho$ is a constant or $\omega$ is a Monge-Amp\`ere form (i.e. $c_{l\b{k}}=0$), then 
\begin{align}\label{2.5}
	\frac{\p G_{j\b{k}}}{\p t^l}=	\langle L_{V_l}\k_j, \k_{k}\rangle=\int_{X_t} (L_{V_l}\k_j)\cdot\o{\k_{k}}\frac{\omega^n_t}{n!}.
	\end{align}
\end{prop}

Now we compute the second variation of the generalized Weil-Petersson metric for a Monge-Amp\`ere fibration. Since $[L_{\b{V}_{m}}, L_{V^l}]=L_{[\b{V}_{m}, V_l]}$ and by (\ref{2.5}), so
\begin{aligns}\label{2.10}
\frac{\p^2 G_{j\b{k}}}{\p t^l\p \b{t}^m}&=\frac{\p}{\p \b{t}^m}	\int_{X_t} (L_{V_l}\k_j)\cdot\o{\k_{k}}\frac{\omega^n_t}{n!}\\
&=\int_{X_t} (L_{\b{V}_m}L_{V_l}\k_j)\cdot\o{\k_{k}}\frac{\omega^n_t}{n!}+\int_{X_t} L_{V_l}\k_j\cdot L_{\b{V}_m}\o{\k_{k}}\frac{\omega^n_t}{n!}\\
&=\int_{X_t} L_{[\b{V}_m, V_l]}\cdot\o{\k_{k}}\frac{\omega^n_t}{n!}+\frac{\p}{\p t^l}\int_{X_t}L_{\b{V}_m}\k_j\cdot\o{\k_{k}}\frac{\omega^n_t}{n!}\\
&\quad-\int_{X_t}L_{\b{V}_m}\k_j\cdot L_{V_l}\o{\k_{k}}\frac{\omega^n_t}{n!}+\int_{X_t} L_{V_l}\k_j\cdot L_{\b{V}_m}\o{\k_{k}}\frac{\omega^n_t}{n!}\\
&=-\int_{X_t}L_{\b{V}_m}\k_j\cdot L_{V_l}\o{\k_{k}}\frac{\omega^n_t}{n!}+\int_{X_t} L_{V_l}\k_j\cdot L_{\b{V}_m}\o{\k_{k}}\frac{\omega^n_t}{n!},
\end{aligns}
  where the last equality holds by (\ref{2.2}) and using \cite[Lemma 2.6]{Sch}, 
  $$[\b{V}_{m},V_l]=-(c_{l\b{m}})^{;\alpha}\frac{\p}{\p\zeta^{\alpha}}+(c_{l\b{m}})^{;\b{\beta}}\frac{\p}{\p\b{\zeta}^{\beta}},$$
which vanishes in the case of Monge-Amp\`ere fibration. 

From (\ref{2.6}), one has
\begin{aligns}\label{2.8}
	\int_{X_t}L_{\b{V}_m}\k_j\cdot L_{V_l}\o{\k_{k}}\frac{\omega^n_t}{n!}&=\int_{X_t}(-\o{\k_{m}}\k_j+\k_j\o{\k_{m}})\cdot(-\k_{l}\o{\k_{k}}+\o{\k_{k}}\k_{l})\frac{\omega^n_t}{n!}\\
	&=-\int_M (\text{Tr}(\o{\k_{m}}\k_j\o{\k_{k}}\k_{l})+\text{Tr}(\k_j\o{\k_{m}}\k_{l}\o{\k_{k}}))\frac{\omega^n_t}{n!}\\
	&=-\langle\o{\k_{m}}\k_j, \o{\k_{l}}\k_k\rangle-\langle\k_j\o{\k_{m}}, \k_k\o{\k_l}\rangle.
\end{aligns}
By (\ref{2.7}) and (\ref{2.13}), one has
\begin{align}\label{2.9}
\int_{X_t} L_{V_l}\k_j\cdot L_{\b{V}_m}\o{\k_{k}}\frac{\omega^n_t}{n!}=\langle L_{V_l}\k_j,  L_{V_m}\k_{k}\rangle.	
\end{align}
Substituting (\ref{2.8}) and (\ref{2.9}) into (\ref{2.10}), we have
\begin{align}\label{2.11}
	\frac{\p^2 G_{j\b{k}}}{\p t^l\p \b{t}^m}=\langle\o{\k_{m}}\k_j, \o{\k_{l}}\k_k\rangle+\langle\k_j\o{\k_{m}}, \k_k\o{\k_l}\rangle+\langle L_{V_l}\k_j,  L_{V_m}\k_{k}\rangle.
\end{align}
Denote by $\operatorname{H}: A^{0,1}(X_t, T_{X_t})\to \text{Span}\{\k_{i}\}$ the orthogonal projection. By Proposition \ref{prop1}, one has
\begin{align}\label{2.12}
G^{p\b{q}}\frac{\p G_{j\b{q}}}{\p t^l}\frac{\p G_{p\b{k}}}{\p\b{t}^m}=G^{p\b{q}}\langle L_{V_l}\k_j,\k_q\rangle \langle\k_p,L_{V_m}\k_k\rangle=\langle\operatorname{H}(L_{V_l}\k_j),\operatorname{H}(L_{V_m}\k_k)\rangle.	
\end{align}
From (\ref{2.11}) and (\ref{2.12}), we obtain
\begin{thm}\label{thm1}
The curvature of generalized Weil-Petersson metric $\omega_{\mc{WP}}$ for a Monge-Amp\`ere fibration  is 
\begin{align*}
R_{j\b{k}l\b{m}}&=-\frac{\p^2 G_{j\b{k}}}{\p t^l\p \b{t}^m}+	G^{p\b{q}}\frac{\p G_{j\b{q}}}{\p t^l}\frac{\p G_{p\b{k}}}{\p\b{t}^m}\\
&=-\langle\o{\k_{m}}\k_j, \o{\k_{l}}\k_k\rangle-\langle\k_j\o{\k_{m}}, \k_k\o{\k_l}\rangle-\langle\operatorname{H}^{\perp}(L_{V_l}\k_j),  \operatorname{H}^{\perp}(L_{V_m}\k_{k})\rangle.
\end{align*}
	Here $\operatorname{H}^{\perp}$ denotes the orthogonal projection from $A^{0,1}(X_t, T_{X_t})$ to  $\operatorname{Span}\{\k_{i}\}^{\perp}$.
\end{thm}
\begin{rem}
For a general relative K\"ahler fibration, we can also obtain the curvature of generalized Weil-Petersson metric  $\left\langle\cdot,\cdot\right\rangle_{\mc{WP}}$. For more details, one can refer to  \cite[Section 4]{WW}.	
\end{rem}

For any two vectors $\xi=\xi^j\frac{\p}{\p t^j},\eta=\eta^j\frac{\p}{\p t^j}$ in $T_t\mc{B}$, we denote 
\begin{align*}
	\k_{\xi}=\k_j\xi^j,\quad \k_{\eta}=\k_j\eta^j.
\end{align*}
From Theorem \ref{thm1}, the holomorphic bisectional curvature satisfies
\begin{align}\label{2.15}
\begin{split}
R(\xi,\o{\xi},\eta,\o{\eta})&:=R_{j\b{k}l\b{m}}\xi^j\b{\xi}^k\eta^l\b{\eta}^m\\
&\leq -\langle\o{\k_{\eta}}\k_\xi, \o{\k_{\eta}}\k_\xi\rangle-\langle\k_\xi\o{\k_{\eta}}, \k_\xi\o{\k_\eta}\rangle\\
&=-2\langle\o{\k_{\eta}}\k_\xi, \o{\k_{\eta}}\k_\xi\rangle.	
\end{split}
\end{align}
Note that 
\begin{align}\label{2.17}
	\langle\o{\k_{\eta}}\k_\xi, \o{\k_{\eta}}\k_\xi\rangle\geq  \frac{1}{n}\left|\sum_{\beta=1}^n(\k_{\eta}\o{\k_{\xi}})^{\beta}_{\beta}\right| ^2=\frac{1}{n}\left|\text{Tr}(\k_{\eta}\o{\k_{\xi}})\right|^2.
\end{align}
In fact, by taking a normal coordinate system around a fixed point, one can assume that $\phi_{\alpha\b{\beta}}=\delta_{\alpha\beta}$ at this point. Hence
\begin{align*}
	\langle\o{\k_{\eta}}\k_\xi, \o{\k_{\eta}}\k_\xi\rangle&=(\o{\k_{\eta}}\k_\xi)^{\b{\gamma}}_{\b{\beta}}(\k_{\eta}\o{\k_{\xi}})^{\tau}_{\alpha}\phi^{\alpha\b{\beta}}\phi_{\tau\b{\gamma}}\\
	&=\sum_{\beta,\gamma=1}^n(\o{\k_{\eta}}\k_\xi)^{\b{\gamma}}_{\b{\beta}}(\k_{\eta}\o{\k_{\xi}})^{\gamma}_{\beta}\geq \sum_{\beta=1}^n|(\k_{\eta}\o{\k_{\xi}})^{\beta}_{\beta}|^2\\
	&\geq \frac{1}{n}\left(\sum_{\beta=1}^n|(\k_{\eta}\o{\k_{\xi}})^{\beta}_{\beta}|\right)^2\geq \frac{1}{n}\left|\sum_{\beta=1}^n(\k_{\eta}\o{\k_{\xi}})^{\beta}_{\beta}\right| ^2=\frac{1}{n}\left|\text{Tr}(\k_{\eta}\o{\k_{\xi}})\right|^2.
\end{align*}
By (\ref{2.17}), we have
\begin{aligns}\label{2.14}
\langle\o{\k_{\eta}}\k_\xi, \o{\k_{\eta}}\k_\xi\rangle&=\int_{X_t}\langle\o{\k_{\eta}}\k_\xi, \o{\k_{\eta}}\k_\xi\rangle\frac{\omega^n_t}{n!}\\
&	\geq \int_{X_t}\frac{1}{n}\left|\text{Tr}(\k_{\eta}\o{\k_{\xi}})\right|^2\frac{\omega^n_t}{n!}\\
&\geq \frac{1}{n}\left(\int_{X_t}\left|\text{Tr}(\k_{\eta}\o{\k_{\xi}})\right|\frac{\omega^n_t}{n!}\right)^2\left(\int_{X_t}\frac{\omega^n_t}{n!}\right)^{-1}\\
&\geq \frac{1}{n}|\langle\eta,\xi\rangle_{\mc{WP}}|^2|X_t|^{-1},
\end{aligns}
where $|X_t|:=\int_{X_t}\frac{\omega^n_t}{n!}$ denotes the volume of each fiber.  From (\ref{2.15}) and (\ref{2.14}), we obtain
\begin{align}\label{2.16}
	R(\xi,\o{\xi},\eta,\o{\eta})\leq -\frac{2}{n}|X_t|^{-1}|\langle\eta,\xi\rangle_{\mc{WP}}|^2.
\end{align}
From (\ref{2.16}), we obtain the holomorphic bisectional curvature of the generalized Weil-Petersson metric is non-positive, and is negative if $\xi$ and $\eta$ are not orthogonal to each other. The holomorphic sectional curvature satisfies
\begin{align*}
	\frac{R(\xi,\o{\xi},\xi,\o{\xi})}{\|\xi\|^4}\leq -\frac{2}{n}|X_t|^{-1}.
\end{align*}
The Ricci curvature satisfies
\begin{align*}
\frac{\operatorname{Ric}(\xi,\o{\xi})}{\|\xi\|^2} &=\frac{\sum_{j=1}^{\dim\mc{B}} R(\xi,\o{\xi},e_j,\o{e_j})}{\|\xi\|^2}\\
&\leq -\frac{2}{n}|X_t|^{-1}\frac{\sum_{j=1}^{\dim\mc{B}} |\langle e_j,\xi\rangle_{\mc{WP}}|^2}{\|\xi\|^2}=-\frac{2}{n}|X_t|^{-1},
\end{align*}
where $\{e_j\}$ is an orthonormal basis with respect to the generalized Weil-Petersson metric. The scalar curvature satisfies
\begin{align*}
\sum_{j=1}^{\dim\mc{B}}\operatorname{Ric}(e_j,\o{e_j})\leq  -\frac{2}{n}|X_t|^{-1}\dim\mc{B}.	
\end{align*}
In a word, we obtain
\begin{cor}
For a Monge-Amp\`ere fibration $p:(\mc{X},\omega)\to \B$, the holomorphic bisectional curvature of generalized Weil-Petersson metric $\omega_{\mc{WP}}$ satisfies 
	\begin{align*}
	R(\xi,\o{\xi},\eta,\o{\eta})\leq -\frac{2}{n}|X_t|^{-1}|\langle\eta,\xi\rangle_{\mc{WP}}|^2.
\end{align*}
for any two vectors $\eta,\xi$ in $T_t\mc{B}$, where $|X_t|:=\int_{X_t}\frac{\omega^n_t}{n!}$ denotes the volume of each fiber. In particular, 
\begin{itemize}
\item[(i)] Holomorphic bisectional curvature is non-positive, and is negative if $\langle\eta,\xi\rangle_{\mc{WP}}\neq 0$;
\item[(ii)] Holomorphic sectional curvature and Ricci curvature are both bounded from above by $-\frac{2}{n}|X_t|^{-1}$, the scalar curvature is bounded from above by $-\frac{2}{n}|X_t|^{-1}\dim\mc{B}$.
\end{itemize}
\end{cor}

\section{Existence of Monge-Amp\`ere fibrations}\label{Sec3}

In this section, we will discuss some existence results on the Monge-Amp\`ere fibrations.

\subsection{Projectively flat vector bundles}\label{sec3.1}

From  \cite[Corollary 1.2.7, Proposition 1.2.8]{Ko3}, a complex vector bundle $E$ is projectively flat if it admits a projectively flat connection, i.e. the curvature satisfies 
\begin{align}\label{3.1}
R=\alpha \text{Id}_E	
\end{align}
for some $2$-form $\alpha$. For a holomorphic Hermitian vector bundle $(E, h)$,  it is called projectively flat if the Chern curvature of $h$ satisfies (\ref{3.1}) for some $(1,1)$-form $\alpha$ (see e.g. the proof of \cite[Proposition 4.1.11]{Ko3} ).
\begin{defn}
 Let $\pi: E\to \mc{B}$ be a holomorphic vector bundle of rank $r$ over a complex manifold $\mc{B}$, we say that the holomorphic vector bundle $E$ admits a {\it projectively flat Hermitian structure} if there exists a Hermitian metric $h$ such that $(E,h)$ is projectively flat. 
\end{defn}
Let $\{s_\alpha\}_{\alpha=1}^r$ denote a local holomorphic frame of $E$, $r=\operatorname{rank}E$, and $\{s^{\alpha}\}_{\alpha=1}^r$ denote the dual frame of $\{s_{\alpha}\}$, $h_{\alpha\b{\beta}}:=h(s_{\alpha}, s_{\beta})$ and $(h^{\b{\beta}\alpha})$ be the inverse matrix of $(h^{\b{\beta}\alpha})$. Then the Chern curvature is given by
\begin{align*}
R&=R^{\alpha}_{\beta j\b{k}}s_{\alpha}\otimes s^{\beta}\otimes dt^j\wedge d\b{t}^k\\
&=	h^{\b{\gamma}\alpha}R_{\beta\b{\gamma}j\b{k}}s_{\alpha}\otimes s^{\beta}\otimes dt^j\wedge d\b{t}^k\\
&=h^{\b{\gamma}\alpha}(-\p_j\p_{\b{k}}h_{\beta\b{\gamma}}+\p_jh_{\beta\b{\sigma}}\p_{\b{k}}h_{\tau\b{\gamma}}h^{\b{\sigma}\tau})s_{\alpha}\otimes s^{\beta}\otimes dt^j\wedge d\b{t}^k\in A^{1,1}(\B,\operatorname{End}E).
\end{align*}
The Ricci curvature is given by
\begin{align*}
	\text{Ric}:=\text{Tr}R=\b{\p}\p\log\det h,
\end{align*}
which is a $d$-closed $(1,1)$-form on $\mc{B}$. If $(E,h)$ is projectively flat, i.e. it satisfies (\ref{3.1}), by taking trace to both sides of (\ref{3.1}), then 
$
\alpha=\frac{1}{r}\text{Ric}. 	
$
Thus, $(E,h)$ is projectively flat if and only if 
\begin{align}\label{3.2}
R=\frac{1}{r}\text{Ric}\cdot\text{Id}_E.	
\end{align}
Let $P(E):=(E-\{0\})/\mb{C}^*$ be the projectivization of the vector bundle $E$, and consider the projective bundle fibration $p: P(E)\to\mc{B}$. 
\begin{prop}\label{prop2}
If $\pi:E\to\mc{B}$ admits a projectively flat Hermitian structure, then 
$p:P(E)\to \mc{B}$ is a Monge-Amp\`ere fibration.	
\end{prop}
\begin{proof}
With respect to the local frame $\{s_{\alpha}\}_{\alpha=1}^r$ of $E$, we denote by 
\begin{align*}
(t; v)=(t^1,\cdots, t^{\dim \mc{B}}; v^1,\cdots, v^r)	
\end{align*}
the local holomorphic coordinates of the complex manifold $E$, which represents the point $v^{\alpha}s_{\alpha}\in E$. Then one can define a norm on $E$ by
\begin{align*}
	H(v):=h(v^{\alpha}s_{\alpha},v^{\beta}s_{\beta})=h_{\alpha\b{\beta}}v^{\alpha}\b{v}^{\beta}. 
\end{align*}
 From \cite[Lemma 1.3]{FLW}, one has 
 \begin{align}\label{3.3}
 \p\b{\p}\log H=-R_{\alpha\b{\beta}j\b{k}}\frac{v^{\alpha}\b{v}^{\beta}}{H}dz^j\wedge d\b{z}^k+\frac{\p^2\log H}{\p v^{\alpha}\p\b{v}^{\beta}}\delta v^{\alpha}\wedge \delta \b{v}^{\beta},	
 \end{align}
where $\delta v^{\alpha}:=dv^{\alpha}+v^{\beta}h^{\b{\gamma}\alpha}\p_j h_{\beta\b{\gamma}}dt^j$. By condition, $(E,h)$ is projectively flat, i.e. it satisfies (\ref{3.2}), so 
\begin{align}\label{3.4}
R_{\alpha\b{\beta}j\b{k}}dz^j\wedge d\b{t}^k=\frac{1}{r}\text{Ric}\cdot h_{\alpha\b{\beta}}.	
\end{align}
Substituting (\ref{3.4}) into (\ref{3.3}), one has 
\begin{align}\label{3.5}
\p\b{\p}\log H=-\frac{1}{r}\text{Ric}+\frac{\p^2\log H}{\p v^{\alpha}\p\b{v}^{\beta}}\delta v^{\alpha}\wedge \delta \b{v}^{\beta}.
\end{align}
Now we define the following $d$-closed real $(1,1)$-form on $P(E)$ by
\begin{align*}
\omega:=i(\p\b{\p}\log H+\frac{1}{r}\text{Ric}).
\end{align*}
Then $\omega$ is a relative K\"ahler form. Indeed,   for any $t\in \mc{B}$, 
 by taking a normal coordinates system around  $t$, $h_{\alpha\b{\beta}}(t)=\delta_{\alpha\b{\beta}}$, then 
\begin{align*}
\omega|_{P(E_{t})}=i(\p\b{\p}\log H+\frac{1}{r}\text{Ric})|_{P(E_{t})}=i\p\b{\p}\log \sum_{\alpha=1}^r|v^{\alpha}|^2>0,
\end{align*}
which is exactly the Fubini-Study metric on $P(E_{t})=\mb{P}^{r-1}$, so we conclude that $\omega$ is relative K\"ahler. From (\ref{3.5}), one has
\begin{align*}
\omega=	i\frac{\p^2\log H}{\p v^{\alpha}\p\b{v}^{\beta}}\delta v^{\alpha}\wedge \delta \b{v}^{\beta},
\end{align*}
 which vanishes along the tautological direction, i.e. $\frac{\p^2\log H}{\p v^{\alpha}\p\b{v}^{\beta}}v^{\alpha}\b{v}^{\beta}=0$. It follows that $\omega^r=0$. Thus $\omega$ is a Monge-Amp\`ere form, and $p: P(E)\to \mc{B}$ is a Monge-Amp\`ere fibration.
\end{proof}
 Let $p:(P(E),\omega)\to \mc{B}$ be a Monge-Amp\`ere fibration over a compact K\"ahler manifold $\mc{B}$. Denote by $\omega_{\mc{B}}$  a K\"ahler metric on $\mc{B}$, by taking a large $C>0$, one concludes that   $\omega+Cp^*\omega_{\mc{B}}$ is a K\"ahler metric on $P(E)$, so $P(E)$ is a compact K\"ahler manifold. Let $\mc{O}_{P(E)}(1)$ denote the hyperplane line bundle over $P(E)$. Then 
 \begin{prop}
 There exist a constant $k\in \mb{R}$ and a $d$-closed 	real $(1,1)$-form $\alpha$ on $\mc{B}$ such that 
 \begin{align}
[\omega]=kc_1(\mc{O}_{P(E)}(1))+[p^*\alpha].
\end{align}
Here $[\bullet]$ denotes the de Rham cohomology class. 
 \end{prop}
 \begin{proof}
Note that the de Rham cohomology class of $P(E)$ satisfies 
 \begin{align*}
 H^*_{dR}(P(E),\mb{R})=H^*_{dR}(\mc{B},\mb{R})[x]/(x^r + c_1(E)x^{r-1} +\cdots+ c_r(E)),	
 \end{align*}
where $x=c_1(\mc{O}_{P(E)}(1))$ (see e.g. \cite[(20.7)]{BT}), so 
\begin{align*}
H^2_{dR}(P(E),\mb{R})=p^*H^2_{dR}(\mc{B}, \mb{R})\oplus \mb{R}x.	
\end{align*}
Let $H^*_{dR}(P(E),\mb{C})$ denote the de Rham cohomology with complex coefficients. By Hodge decomposition theorem (see e.g. \cite[Theorem 5.1]{Wells}), one has
\begin{align*}
	H^2_{dR}(P(E),\mb{C})=H^{2,0}_{\b{\p}}(P(E))\oplus H^{1,1}_{\b{\p}}(P(E)) \oplus H^{0,2}_{\b{\p}}(P(E))
\end{align*}
where $H^{*,*}_{\b{\p}}(P(E))$ denotes the Dolbeault cohomology.
Since $x\in H^2_{dR}(P(E),\mb{R})\cap H^{1,1}_{\b{\p}}(P(E))$,  so 
\begin{align*}
H^2_{dR}(P(E),\mb{R})\cap H^{1,1}_{\b{\p}}(P(E))&=p^*H^2_{dR}(\mc{B}, \mb{R})\cap H^{1,1}_{\b{\p}}(P(E))\oplus \mb{R}x\\
&=p^*(H^2_{dR}(\mc{B}, \mb{R})\cap H^{1,1}_{\b{\p}}(\mc{B}))\oplus \mb{R}x,
\end{align*}
where the last equality follows from the Hodge decomposition theorem for 
the compact K\"ahler manifold $\mc{B}$.
Since  $[\omega]\in H^2_{dR}(P(E),\mb{R})\cap H^{1,1}_{\b{\p}}(P(E))$, and note that any element in  $H^2_{dR}(\mc{B}, \mb{R})\cap H^{1,1}_{\b{\p}}(\mc{B})$ is represented by a  $d$-closed real $(1,1)$-form on $\mc{B}$, so 
 \begin{align*}
[\omega]=kx+[p^*\alpha]=kc_1(\mc{O}_{P(E)}(1))+[p^*\alpha]
\end{align*}
for some $k\in \mb{R}$ and some $d$-closed real $(1,1)$-form $\alpha$ on $\mc{B}$. 
 \end{proof}

Since $\omega$ is a relative K\"ahler form, so $k>0$.  By the $\p\b{\p}$-lemma for compact K\"ahler manifolds (see e.g. \cite[Proposition 1.7.24]{Ko3}), there exists a metric  $e^{-\psi}$ on $\mc{O}_{P(E)}(1)$ such that its curvature satisfies 
\begin{align*}
i\p\b{\p}\psi=\frac{1}{k}(\omega-p^*\alpha).
\end{align*}
By the condition $\omega^r=0$, the geodesic curvature form $c(\psi)$ satisfies 
\begin{align}\label{3.7}
 c(\psi):=c(i\p\b{\p}\psi)=-\frac{1}{k}p^*\alpha.	 
 \end{align}
Now we denote
\begin{align}\label{3.6}
	L:=\mc{O}_{P(E)}(1)\otimes K_{P(E)/\mc{B}}^{-1}=\mc{O}_{P(E)}(r+1)\otimes p^*\det E,
\end{align}
where the second equality follows from \cite[Proposition 2.2]{Kob1}. Since 
\begin{align*}
	c_1(\det E)=-p_*\left(c_1(\mc{O}_{P(E)}(1))^r\right)
\end{align*}
(see e.g. \cite[Section 3.2]{Fulton}), so there exists a metric $h_1$ on $\det E$ such that 
\begin{align}\label{3.8}
c_1(\det E, h_1)=-\int_{P(E)/\mc{B}}\left(\frac{i}{2\pi}\p\b{\p}\psi	\right)^r=-\frac{r}{(2\pi)^r}\int_{X_t}c(\psi)(i\p\b{\p}\psi)^{r-1}_{|X_t}=\frac{r\alpha}{2\pi k},
\end{align}
where the last equality follows from (\ref{3.7}) and noting $\int_{X_t}(\frac{i}{2\pi}\p\b{\p}\psi)^{r-1}_{|_{X_t}}=1$.
From  (\ref{3.6}), the induced metric on $L$ is 
\begin{align*}
e^{-\phi}=e^{-(r+1)\psi}\cdot p^*h_1. 	
\end{align*}
The curvature of $e^{-\phi}$ is 
\begin{align}\label{3.9}
\p\b{\p}\phi=(r+1)\p\b{\p}\psi+p^*\b{\p}\p\log h_1.	
\end{align}
By (\ref{3.7}), (\ref{3.8}) and (\ref{3.9}), one has
\begin{aligns}\label{3.10}
c(\phi) &=(r+1)c(\psi)+ip^*\b{\p}\p\log h_1	\\
&=(r+1)(-\frac{1}{k}p^*\alpha)+2\pi  p^*c_1(\det E, h_1)\\
&=-\frac{1}{k}p^*\alpha.
\end{aligns}
By \cite[Lemma 5.37]{Shiff}, one knows that
\begin{align*} 
E^*=p_*(\mc{O}_{P(E)}(1))=p_*(L\otimes K_{P(E)/\mc{B}}).	
\end{align*}
Following Berndtsson (cf. \cite{Bern2,  Bern4}), one can define the following $L^2$-metric on the direct image bundle $E^*$: for any $u\in E_{t}^*\equiv H^0(X_t, (L\otimes K_{P(E)/\mc{B}})|_{X_t})$, $t\in \mc{B}$, then
\begin{align}\label{L2 metric}
\|u\|^2=\int_{X_t}|u|^2e^{-\phi}. 	
\end{align}
Note that $u$ can be written locally as $u=f dv\wedge e=fdv^1\wedge \cdots\wedge dv^n\otimes e$, where $e$ is a local holomorphic frame for $L|_{X_t}$, and so locally
$$|u|^2e^{-\phi}:=i^{n^2}|f|^2 |e|^2dv\wedge d\b{v}=i^{n^2}|f|^2 e^{-\phi}dv\wedge d\b{v}.$$
\begin{thm}[{\cite[Theorem 1.2]{Bern4}}]
\label{thm4} For any $t\in \mc{B}$ and let $u\in E_{t}^*$, one has
\begin{align}\label{cur}
\langle iR^{E^*}u,u\rangle=\int_{X_t}c(\phi)|u|^2e^{-\phi}+\langle(1+\Box')^{-1}\k_j\cdot u,\k_k\cdot u\rangle i dt^j\wedge d\b{t}^k,
\end{align}
where $R^{E^*}$ denotes the curvature of the Chern connection on $E^*$ with respect to the $L^2$ metric defined above, here $\Box'=\n'\n'^*+\n'^*\n$ is the Laplacian on $L|_{X_t}$-valued forms on $X_t$ defined by the $(1,0)$-part of the Chern connection on $L|_{X_t}$.
\end{thm}
Let $\{u_\alpha\}, 1\leq \alpha\leq r$, be a local holomorphic frame of $E^*$, and  set
$$G_{\alpha\b{\beta}}=\langle u_\alpha, u_\beta\rangle=\int_{X_t}u_\alpha \o{u_\beta}e^{-\phi}.$$
By taking trace to both sides of (\ref{cur}) and using (\ref{3.10}), we have 
\begin{align}\label{3.11}
i\text{Ric}^{E^*}=-\frac{r}{k}\alpha+\langle(1+\Box')^{-1}\k_j\cdot u_\alpha,\k_k\cdot u_\beta\rangle G^{\alpha\b{\beta}} i dt^j\wedge d\b{t}^k \geq -\frac{r}{k}\alpha,
\end{align}
where the above equality holds if and only if $\k_j=0$ for all $1\leq j\leq \dim\mc{B}$. From (\ref{3.8}), one has
 \begin{align}\label{3.12}
 [i\text{Ric}^{E^*}]=2\pi c_1(E^*)=\left[-\frac{r}{k}\alpha\right].	
 \end{align}
Combining (\ref{3.11}) with (\ref{3.12}) shows that  $i\text{Ric}^{E^*}= -\frac{r}{k}\alpha$ and thus
\begin{align}\label{3.13}
	\k_j\equiv 0
\end{align}
on $P(E)$. Since the generalized Weil-Petersson metrics with respect to $\omega$ and $i\p\b{\p}\phi$ are the same, so $\omega_{\mc{WP}}\equiv 0$ on $\mc{B}$. Substituting (\ref{3.13}) into (\ref{cur}), we get
\begin{align*}
\langle i R^{E^*}u,u\rangle=\int_{X_t}c(\phi)|u|^2e^{-\phi}=-\frac{\alpha}{k}\|u\|^2,
\end{align*}
which is equivalent to $ R^{E^*}=i\frac{\alpha}{k}\text{Id}_{E^*}$. Thus, with respect to the dual metric of the $L^2$-metric (\ref{L2 metric}), the Chern curvature $R^E$ is given by
\begin{align*}
R^E=-i\frac{\alpha}{k}\text{Id}_{E}, 	
\end{align*}
which implies that $E$ is projectively flat. 
\begin{thm}\label{thm2}
If $p:(P(E),\omega)\to \mc{B}$ is a Monge-Amp\`ere fibration over a compact K\"ahler manifold $\mc{B}$, then $E$ admits a projectively flat Hermitian structure, and $\omega_{\mc{WP}}\equiv 0$ on $\mc{B}$. 	
\end{thm}

From \cite[(2.3.4), (2.3.5) and Proposition 2.3.1 (b)]{Ko3}, we obtain
\begin{cor}
If $p:P(E)\to \mc{B}$ is a Monge-Amp\`ere fibration over a compact K\"ahler manifold $\mc{B}$, then
\begin{itemize}
\item[(i)] $c(E)=\left(1+\frac{c_1(E)}{r}\right)^r$;
\item[(ii)] $\operatorname{ch}(\operatorname{End}(E))=r^2$.
\end{itemize}
\end{cor}
For the case of $\mc{B}$ is a compact Riemann surface, $\dim \mc{B}=1$. Put
\begin{align*}
\mu(E)=\frac{\int_{\mc{B}}c_1(E)}{\text{rank}(E)}.	
\end{align*}
Recall that $E$ is said to be stable (resp. semi-stable) in the sense of Mumford if for every proper subbundle $E'$ of $E$, $0<\text{rank}(E')<\text{rank}(E)$, we have 
\begin{align}
\mu(E')<\mu(E),\quad (resp. \  \ \mu(E')\leq \mu(E)).	
\end{align}
 $E$ is called polystable if $E=\oplus E_i$ with $E_i$ stable vector bundles all of the same slope $\mu(E)=\mu(E_i)$, see e.g. \cite[Section 4.B]{Huy}.   Thus

\begin{thm}\label{thm6.1}  Let $E$ be a holomorphic vector bundle over a compact K\"ahler manifold $\B$.  Let $P (E):= (E-\{0\})/\mathbb C^*$ be the projectivization of $E$. Then the following are equivalent:
\begin{itemize}
\item[1)]$E$ admits a projectively flat Hermitian structure;
\item[2)] $p:P(E)\to \B$ is a Monge-Amp\`ere fibration.
\end{itemize}
For the case of $\dim \B=1$, both are equivalent to the polystability of  $E$.
\end{thm}
 \begin{proof}
 Now it suffices to prove the last part. Assume that  $\dim \mc{B}=1$, i.e. $\mc{B}$ is a compact Riemann surface. 
By \cite[Proposition 5.2.3]{Ko3}, $(E,h)$ is projectively flat if and only if $(E,h)$ is weak Hermitian-Einstein , i.e. 
  $\Lambda_{\omega_{\mc{B}}}R^E=\varphi \text{Id}_E$ for some function $\varphi$. By a conformal change (see e.g. \cite[Proposition 4.2.4]{Ko3}), $E$ admits a weak Hermitian-Einstein metric if and only if $E$ admits a Hermitian-Einstein metric. Thus, $E$ admits a Hermitian-Einstein metric if and only if $E$ admits a projectively flat Hermitian metric,  
 which is equivalent to that  $p: P(E)\to \mc{B}$ is a Monge-Amp\`ere fibration. All are equivalent to the polystability of $E$ (see e.g. \cite[Theorem 4.B.9]{Huy}). The proof is complete.	
 \end{proof}

\begin{rem}
In \cite{Aikou}, T. Aikou considered the projectively flat holomorphic vector bundle from the view of complex Finsler geometry, and proved that $E$ admits a projectively flat Hermitian metric if and only if the projective bundle $p:P(E)\to \mc{B}$ is a flat K\"ahler fibration (see \cite[ Theorem 3.2]{Aikou}), where 
a K\"ahler fibration \(p: \mathcal{X} \rightarrow \B\) with a smooth family of K\"ahler metrics $\{\Pi_z\}_{z\in \B}$ is said to be flat if, at each point \(z \in \B\),
there exists an open neighborhood \(U\) of \(z\) so that we can choose K\"ahler potentials for \(\Pi_{z}\)
which is independent of \(z \in U \), see \cite[Definition 1.2]{Aikou}. 
 Combining with Proposition \ref{prop1} and Theorem \ref{thm1}, in the case that $\mc{B}$ is a compact K\"ahler manifold, the projective bundle $p:P(E)\to \mc{B}$ is a Monge-Amp\`ere fibration if and only if it is a flat K\"ahler fibration.
\end{rem}

\begin{rem}

After our paper \cite{WW} was submitted to arXiv, by using the negativity of direct image bundles \cite[Section 3]{Bern3},  S. Finski \cite[Theorem 5.1]{Finski} obtained another kind of description  of the projectively flat holomorphic vector bundles, i.e., $E$ admits a projectively flat Hermitian structure if and only if the class $$\Lambda_E:=c_{1}\left(\mathcal{O}_{P\left(E^{*}\right)}(1)\right)-\frac{1}{r } p^{*} c_{1}(E)$$ is semi-positive. In fact, if $E$ admits a projectively flat Hermitian structure, so is $E^*$. By  \cite[Proposition 6.2, (6.6)]{WW}, one knows that
$\Lambda_{E}$
is semi-positive. Conversely, if $\Lambda_E$ is semi-positive,  let $\alpha$ be a semi-positive form in the class $\Lambda_E$, then 
\begin{align*}
\int_{P(E^*)}\alpha^{r}\wedge p^*\omega_0^{m-1}=\int_{P(E^*)} 	\Lambda_E^{r}\wedge p^*\omega_0^{m-1}\geq 0,
\end{align*}
where $\omega_0$ is a K\"ahler form on $\B$, $\dim \mc{B}=m$. On the other hand, 
\begin{align*}
\int_{P(E^*)}	\Lambda_E^{r}\wedge p^*\omega_0^{m-1} &=\int_{P(E^*)}c_{1}\left(\mathcal{O}_{P\left(E^{*}\right)}(1)\right)^r\wedge p^*\omega_0^{m-1}\\
&\quad-\int_{P(E^*)}c_{1}\left(\mathcal{O}_{P\left(E^{*}\right)}(1)\right)^{r-1}\wedge p^*c_1(E)\wedge p^*\omega_0^{m-1}\\
&=\int_{\B}c_1(E)\wedge \omega_0^{m-1}-\int_{\B}c_1(E)\wedge \omega_0^{m-1}=0,
\end{align*}
which follows that $\alpha^r\wedge p^*\omega_0^{m-1}=0$ since $\alpha$ is semi-positive, which is equivalent to $\alpha^r=0$, i.e. $\alpha$ is a Monge-Amp\`ere form. By \cite[Theorem B]{WW} or Theorem \ref{thm6.1}, $E$ admits a projectively flat Hermitian structure.  	
\end{rem}

\subsection{Infinite rank flat Higgs bundles}\label{sec3.2}

Firstly, we will recall the notion of quasi-vector bundles, and one can refer to an early version of \cite{BPW}. 

\begin{defn}[Quasi-vector bundle] Let $A:=\{A_t\}_{t\in \B}$ be a family of $\mathbb C$-vector spaces over a smooth manifold $\B$. Let $\Gamma$ be a $C^{\infty}(\B)$-submodule of the space of all sections of $A$. We call $\Gamma$ a smooth quasi-vector bundle structure on $V$ if each vector of the fiber $A_t$ extends to a section in $\Gamma$ locally near $t$.
\end{defn}

Let $p:(\X,\omega) \to \B$ be a relative K\"ahler fibration. Let $E$ be a holomorphic vector bundle over $\X$ with smooth Hermitian metric $h_E$. We write
$$
X_t:=p^{-1}(t), \ \ E_t:= E_{X_t}, \ \ h_{E_t}:=h_E|_{E_t}.
$$
For each $t\in \B$, denote by $\mathcal A^{p,q}(E_t)$ the space of all smooth $E_t$-valued $(p,q)$-forms on $X_t$. Put
$$
\mathcal A^{p,q}:=\{\mathcal A^{p,q}(E_t)\}_{t\in \B}.
$$
Denote by $\mathcal A^{p,q}(E)$ the space of smooth $E$-valued $(p,q)$-forms on $\X$. Let us define
\begin{align}\label{Gamma}
\Gamma^{p,q}:=\{u: t\mapsto u^t \in \mathcal A^{p,q}(E_t): \exists \ \mathbf{u} \in \mathcal A^{p,q}(E), \ \mathbf u |_{X_t}=u^t, \ \forall \ t\in \B\}.
\end{align}
We call $\mathbf u$ a \emph{smooth representative} of $u\in \Gamma^{p,q}$. Since $p$ is a proper smooth submersion, we know that each $\Gamma^{p,q}$ defines a quasi-vector bundle structure on $\mathcal A^{p,q}$. Consider
$$
(\mathcal A^k, \Gamma^k):=\oplus_{p+q=k} (\mathcal A^{p,q}, \Gamma^{p,q}).
$$
We know that the fiber of $\mathcal A^k$ can be written as
$$
\mathcal A^k(E_t)=\oplus_{p+q=k} \mathcal A^{p,q}(E_t),
$$
which is the space of all $E$-valued smooth $k$-forms on $X_t$. For every $u\in \Gamma^k$, let us define
\begin{align}\label{Lie derivative connection}
\nabla u:=\sum dt^j \otimes [d^E, \delta_{V_j}] \mathbf u+ \sum d\bar t^j \otimes [d^E, \delta_{\bar V_j}] \mathbf u,
\end{align}
where each $V_j$ denotes the horizontal lift of  $\partial /\partial t^j$ with respect to $\omega$ and
$$
d^E:=\dbar+\partial^E,
$$
denotes the Chern connection on $(E, h_E)$.

\begin{defn}  In this paper we shall identify $u$ with its smooth representative $\mathbf u$. We call $\nabla$ the Lie derivative connection on $(\mathcal A^k, \Gamma^k)$ with respect to $\omega$.
\end{defn}

For each $p,q$ with $p+q=k$, $\nabla$ induces a connection, say $D$, on  $(\mathcal A^{p,q}, \Gamma^{p,q})$. For bidegree reason, we have
\begin{align}\label{Chern connection}
D u:=\sum dt^j \otimes [\partial^E, \delta_{V_j}] \mathbf u+ \sum d\bar t^j \otimes [\dbar, \delta_{\bar V_j}] \mathbf u, \qquad \forall \ u\in \Gamma^{p,q}.
\end{align}
The associated second fundamental form can be written as
$$
(\nabla-D) u=\sum dt^j \otimes \kappa_j \cdot \mathbf u+ \sum d\bar t^j \otimes \overline{\kappa_j}\cdot \mathbf u,
$$
where each 
$$
\kappa_j : \mathbf u\mapsto \kappa_j \cdot \mathbf u,
$$ 
denotes the action of the Kodaira--Spencer tensor $\kappa_j$ on $u$.

\begin{defn} We call 
$$
\theta:= \sum dt^j \otimes \kappa_j,
$$
the Higgs field associated to $(\mathcal A^k, \Gamma^k, \omega)$.
\end{defn}

By Theorem 5.6 in \cite{Wang17} (or an early version of \cite{BPW}), we know that

\begin{prop}\label{th:122} $D$ defines a Chern connection on each $(\mathcal A^{p,q}, \Gamma^{p,q})$ and each $\overline{\kappa_j}=\kappa_j^*$.
\end{prop}

The curvature of the Lie derivative connection is 
\begin{equation}\label{eq:curvature-A}
\nabla^2 u =\sum (dt^j \wedge d\bar t^k)\otimes  [ [d^E, \delta_{V_j}],  [d^E, \delta_{\bar V_k}]] \mathbf u.
\end{equation}
For bidegree reason, it gives the following curvature formula for the induced Chern connection
\begin{equation}\label{eq:curvature-C}
D^2 u= \nabla^2 u- \sum (dt^j \wedge d\bar t^k)\otimes [\kappa_j, \overline{\kappa_{k}}]\cdot  \mathbf u.
\end{equation}
Together with the following Lie derivative identity (see Proposition 4.2 in \cite{Wang15})
\begin{equation}\label{eq:curvature-B}
 [ [d^E, \delta_{V_j}],  [d^E, \delta_{\bar V_k}]] \mathbf u= [d^E, \delta_{[V_j, \bar V_k]}]  \mathbf u +\Theta^E(V_j ,\bar V_k) \mathbf u,
\end{equation}
where $
\Theta^E:=(d^E)^2$
denotes the Chern curvature of $(E, h_E)$, \eqref{eq:curvature-C} and  \eqref{eq:curvature-B} imply

\begin{thm}\label{th:curvature} For every $u\in \Gamma^{p,q}$, write
$$
D^2 u=\sum (dt^j \wedge d\bar t^k)\otimes  \Theta_{j\bar k} u,
$$
then the Chern curvature operators $\Theta_{j\bar k}$ satisfy 
$$
(\Theta_{j\bar k} u, u)=( [d^E, \delta_{[V_j, \bar V_k]}]  \mathbf u, u)+(\Theta^E(V_j ,\bar V_k) \mathbf u, u)+ (\kappa_j u, \kappa_k u)-(\overline{\kappa_{k}} u, \overline{\kappa_{j}}u).
$$
\end{thm}

\begin{prop}\label{pr: flat-higgs} Let $p:(\X,\omega) \to \B$ be a Monge-Amp\`ere fibration. If $\Theta^E\equiv 0$  then
\begin{itemize}
\item [i)] $\nabla^2=0$;
\item [ii)] $\theta^2=0$;
\item [iii)] $D\theta+\theta D=0$. 
\end{itemize}
In particular, each  $(\mathcal A^k, \Gamma^k, D, \theta)$ is an infinite rank flat Higgs bundle.
\end{prop}

\begin{proof} Since the total degree of the Kodaira--Spencer tensor is zero, $\theta^2=0$ is always true. Moreover
$$
D^{1,0}\theta+\theta D^{1,0}=0
$$
follows from $[V_j, V_k]\equiv 0$, which is true for every relative K\"ahler fibration.  Assume further that $\omega$ is a Monge-Amp\`ere form, then we have
$$
[V_j, \overline{V_k}]\equiv 0
$$
by Proposition \ref{pr:rkf}, which gives
$$
D^{0,1}\theta+\theta D^{0,1}=0  \ \ \text{i.e.} \ \theta \ \text{is holomorphic},
$$
and (by \eqref{eq:curvature-B} and \eqref{eq:curvature-A})
$$
\nabla^2= \sum (dt^j \wedge d\bar t^k)\otimes \Theta^E(V_j, \overline{V_k}).
$$
Thus $\nabla^2=0$ if one further assumes that $\Theta^E\equiv 0$.
\end{proof}

\begin{thm}
A relative K\"ahler fibration $p: (\X,\omega) \to \B$ is a Monge-Amp\`ere fibration if and only if the following associated infinite rank Higgs bundle 
$$
(\mc{A},\Gamma, D,\theta)
$$
is Higgs-flat, where each fiber $\mathcal A_t$ denotes the space of smooth differential forms on $X_t$.\end{thm}
\begin{proof}
By taking $E$ to be a trivial bundle, then 
 the bundle $\mathcal A$ is precisely $\oplus_{k=0}^{2n} \mathcal A^k$. Thus if $p: (\X, \omega) \to \B$ is a Monge-Amp\`ere fibration, then Proposition \ref{pr: flat-higgs} implies that $\mathcal A$ is Higgs flat. On the other hand, since
$$
\nabla^2  =\sum (dt^j \wedge d\bar t^k)\otimes   [d, \delta_{[V_j, \overline{V_k}]}], 
$$
we know that if  $\mathcal A$ is Higgs flat, then  $\nabla^2 \equiv 0$ gives
$$
[d, \delta_{[V_j, \overline{V_k}]}] u\equiv 0
$$
on fibers for all smooth form $u$ on $\X$. Take $u$ to be an arbitrary smooth function, we get
$$
[d, \delta_{[V_j, \overline{V_k}]}] u=[V_j, \overline{V_k}] u=0,
$$
which implies $[V_j, \overline{V_k}] \equiv 0$. Thus $\omega$ is a Monge-Amp\`ere form by Proposition \ref{pr:rkf}. The proof is complete.
\end{proof}

\section{Examples of Monge-Amp\`ere fibrations}

In this section, we will introduce some examples of Monge-Amp\`ere fibrations, which are also the motivations for studying such kinds of fibrations.

\subsection{Family of elliptic curves}  For each $t$ in the upper half plane $\mathbb H:=\{t\in\mathbb C: {\rm Im}\, t>0\}$, consider the the following elliptic curve (one dimensional torus) 
$$
X_t:=\mathbb C/{(\mathbb Z+t\mathbb Z)}.
$$
There is a canonical diffeomorphism from each $X_t$ to a fixed elliptic curve, say $X_i$. In fact, the $\mathbb R$-linear quasi-conformal mapping
$
f^t:    \mathbb C\to \mathbb C
$
defined by
\begin{equation}\label{eq:5.1}
f^t(1)=1, \ \ f^t(t)=i,
\end{equation}
naturally induces a map, still denoted by $f^t$, from $X_t$ to $X_i$. A direct computation gives
$$
f^t(\zeta)=z=\frac{i-\bar t}{t-\bar t} \,\zeta+\frac{t-i}{t-\bar t}\, \overline{\zeta}.
$$
Now $\{f^t\}_{t\in\mathbb H}$ defines a smooth trivialization of 
$
\X:=\{X_t\}_{t\in\mathbb H} \simeq (\mathbb H\times \mathbb C)/\mathbb Z^2$
as follows
$$
f: \X \to \mathbb H\times X_i, \ \ \  f(t,\zeta):=(t, f^t(\zeta)).
$$
The natural K\"ahler form  $i \,dz \wedge d\bar z $ on $\mathbb C$ induces a K\"ahler form on $X_i$, thus a relative K\"ahler form, say $\omega_i$ on $\mathbb H\times X_i$. Consider its pull back, say $
\omega:=f^* \omega_i,
$ on $\X$,
we have

\begin{prop} $\omega$ is a Monge-Amp\`ere form on the following canonical  fibration
$$
p: \mathcal X \to \mathbb H, \ \ \ p(X_t):=t.
$$ 
\end{prop}

\begin{proof} Notice that $(i \,dz \wedge d\bar z)^2=0$ gives $\omega^2=0$. Moreover, $\omega$ can be written as the following form: 
$$\omega= i \alpha\wedge\bar \alpha,$$ where
$$
\alpha:=f^* dz= \frac{i-\bar t}{t-\bar t} \,d\zeta+\frac{t-i}{t-\bar t}\, d\bar \zeta + \frac{(i-\bar t)(\bar \zeta-\zeta)}{(t-\bar t)^2}\,dt +\frac{(t-i)(\bar \zeta-\zeta)}{(t-\bar t)^2}\,d\bar t,
$$
we get
$$
\omega=\frac i {{\rm Im}\, t} \left(d\zeta\wedge d\bar \zeta+ A\, d\zeta \wedge d\bar t+ A\, dt \wedge d\bar \zeta+ |A|^2 dt\wedge d\bar t \right), \ \  A:=\frac{\zeta-\bar \zeta}{\bar t-t}.
$$
Thus $\omega$ is of degree-$(1,1)$ and positive on each fiber. Hence $\omega$ is a Monge-Amp\`ere form. 
\end{proof}
\begin{rem}
	The above fibration possesses a natural $SL_2(\mathbb Z)$ action
$$
SL_2(\mathbb Z) \ni 
\begin{pmatrix}
a & b\\
c & d
\end{pmatrix} 
: (t, \zeta) \mapsto \left(\frac{at+b}{ct+d}, \frac{\zeta}{ct+d}\right),
$$
which preserves $\omega$. Let $\Gamma$ be a congruence subgroup of  $SL_2(\mathbb Z)$, then each $\Gamma$ quotient of the upper half-plane $\mathbb H$ can be compactified, thus the regular part induces a Monge-Amp\`ere fibration over a quasi-projective manifold. Similarly, one can also construct the Monge-Amp\`ere family of Abelian varieties, see Remark \ref{remark4.4} for another approach. 
\end{rem}

\subsection{Finite dimensional Higgs bundles} \label{Higgs bundle}
Denote by $\mathfrak{gl}_n (\mathbb C)$ the space of $n$ by $n$ complex matrices. Consider the following bounded symmetric domain of the third type
$$
{\rm BSD_{III}}:=\{B \in \mathfrak{gl}_n (\mathbb C): B=B^T, \ B \bar B^T <1\}, 
$$
where $B^T$ denotes the transpose of $B$ and $B \bar B^T <1$ means all eigenvalues of $B \bar B^T$ are less than one. One may define a canonical  \emph{holomorphic motion} of $\mathbb C^n$: 
\begin{equation}\label{eq:hm}
F:{\rm BSD_{III}} \times \mathbb C^n \to  {\rm BSD_{III}} \times \mathbb C^n; \ \   F(B, z)=(B, \zeta),  \ \ \zeta:=z+B\bar z,
\end{equation}
where we think of $z$ as a column vector and $B\bar z$ denotes the matrix multiplication. The natural metric $i\partial\dbar |z|^2$ on $\mathbb C^n$ defines a relative K\"ahler metric, still write it as $i\partial\dbar |z|^2$,  on ${\rm BSD_{III}} \times \mathbb C^n$.  Then one can check that
$$
\Omega:=(F^{-1})^* (i\partial\dbar |z|^2) 
$$
is of degree $(1,1)$ with respect to the $(B, \zeta)$ coordinate on ${\rm BSD_{III}} \times \mathbb C^n$. 
\begin{thm}\label{th:bo-1}  Put $\X:={\rm BSD_{III}} \times \mathbb C^n$, then the natural projection
$$
p: (B,\zeta) \to B,
$$
defines a (non-proper) Monge-Amp\`ere fibration $p:(\X, \Omega)\to{\rm BSD_{III}} $.
\end{thm}

\begin{proof} Notice that it is positive on the central fiber and symplectic on each fiber, thus $\Omega$ is relative K\"ahler. Moreover, $(i\partial\dbar |z|^2)^n=0$ implies that $\Omega^n=0$. Thus $\Omega$ is a Monge-Amp\`ere form.
\end{proof}
\begin{rem}\label{remark4.4}
	 Fix an abelian variety $\mathbb C^n / \mathbb Z^{2n}$, the map $F$ in \eqref{eq:hm} induces a natural $\mathbb Z^{2n}$ action on $\X$, which gives a Monge-Amp\`ere family of Abelian varieties $\X/\mathbb Z^{2n} \to {\rm BSD_{III}}$.
\end{rem}

\subsubsection{Higgs bundles over ${\rm BSD_{III}}$}For each $t\in {\rm BSD_{III}}$, let us denote by $\mathcal A^k_t$ the space of \emph{translation invariant} $k$-forms on $p^{-1}(t)=\mathbb C^n$. Then we have the following finite rank vector bundle
$$
\mathcal A^k:=\{\mathcal A^k_t\}_{t\in {\rm BSD_{III}}}.
$$
Notice that our holomorphic motion $F$ in \eqref{eq:hm} defines a flat connection
\begin{equation}\label{eq:Bo-nabla}
\nabla:=\sum dt^j \otimes L_{V_j} +\sum d\bar t^j\otimes L_{\bar V_j}, \ \ V_j:=F_*\left(\frac{\partial}{\partial t^j}\right),
\end{equation}
on $\mathcal A^k$ (since $F$ is linear on fibers, the above connection is well defined on the space of invariant forms; flatness follows from $[\frac{\partial}{\partial t^j}, \frac{\partial}{\partial t^k}]=[\frac{\partial}{\partial t^j}, \frac{\partial}{\partial \bar t^k}]=0$). 
Denote by $
\mathcal A^{p,q}:=\{\mathcal A^{p,q}_t\}_{t\in {\rm BSD_{III}}}
$ each $(p,q)$ component of $\mathcal A^k$, i.e. each $\mathcal A^{p,q}_t$ is the space of translation invariant $(p,q)$-forms on $p^{-1}(t)$. By the Cartan formula for the Lie derivative, we have
\begin{equation}\label{eq:Bo3.8}
L_{V_j}=[d, \delta_{V_j}]=[\partial, \delta_{V_j}]+[\dbar, \delta_{V_j}],
\end{equation}
thus only $[\partial, \delta_{V_j}]$ preserve the bidegree, from which we know the induced connection on each $\mathcal A^{p,q}$ can be written as
$$
D=\sum dt^j \otimes D_{\partial/\partial  t^j} + \sum d\bar t^j \otimes D_{\partial/\partial \bar t^k}, \ \  D_{\partial/\partial  t^j}:=[\partial, \delta_{V_j}], \  D_{\partial/\partial \bar t^k}:=[\dbar, \delta_{\bar V_k}],
$$
Moreover, we have
$$
\nabla-D=\theta+\bar \theta, \ \ \theta:=\sum dt^j \otimes [\dbar, \delta_{V_j}].
$$
We call $\theta$ the \emph{Higgs field} on $\mathcal A^k$. We also need the following lemma, which is a special case of Theorem 5.6 in \cite{Wang17}.

\begin{lemma} $D$ defines a Chern connection on each $\mathcal A^{p,q}$ with respect to the metric defined by $\Omega$, moreover $[\dbar, \delta_{V_j}]^*=[\partial, \delta_{\bar V_j}]$.
\end{lemma}

\begin{proof} To show that the $(0,1)$-part of $D$ is integrable, it is enough to prove
$$
[[\dbar, \delta_{\bar V_j}],[\dbar, \delta_{\bar V_k}]]=0,
$$ 
which follows from $[L_{\bar V_j}, L_{\bar V_k}]=L_{[\bar V_j, \bar V_k]}=0$. Now it suffices to check that $D$ preserves the metric and $[\dbar, \delta_{V_j}]^*=[\partial, \delta_{\bar V_j}]$. The idea is to use the primitive decomposition and the fact that $\nabla$ commutes with $\Omega\wedge$. Details can be found in \cite{Wang17}. 
\end{proof}

\begin{thm}\label{th:higgs} The above lemma implies that each $(\mathcal A^k, \theta, D)$ is a flat Hermitian Higgs bundle.
\end{thm}

\subsubsection{Curvature properties of the space of complex structures}

Let $(V, \omega)$ be a $2n$ dimensional real vector space $V$ with a symplectic form $\omega$. Denote by $\mathcal J(V, \omega)$ the space of $\omega$-compatible complex structures on $V$.  For each $J\in {\rm BSD_{III}}$ and $p+q=k$, denote by 
$\wedge_J^{p,q}$ the space of $J$-$(p,q)$-forms in  $\wedge^k (\mathbb C\otimes V^*)$. It is known that  $\mathcal J(V, \omega)$ is isomorphic to ${\rm BSD_{III}}$, and the Higgs bundle $\mathcal A^k$ has the following description
$$
\mathcal A^k \simeq \mathcal H^k:=\oplus_{p+q=k} \mathcal H^{p,q}, \ \  \mathcal H^{p,q}:=\{\wedge_J^{p,q}\}_{J\in\mathcal J(V, \omega)}. 
$$
Thus as in \cite{Wang_Higgs} one may define the associated  \emph{Lu's Hodge metric}, say $\omega_{\mc{WP}, k}$, on $\mathcal J(V, \omega)$.  One may verify that all $\omega_{\mc{WP}, k}$ are equal up to positive constants, i.e.
$$
\omega_{\mc{WP}, k} =c(k,n) \omega_{\mc{WP}, 1},
$$
where $c(k,n)$ depends only on $k$ and $n$. In fact, $\omega_{\mc{WP}, 1}$ is just  the generalized Weil-Petersson metric in Definition  \ref{de:DFWP} (up to a factor).
 Hence 
 $\omega_{\mc{WP},1}$ is K\"ahler on $\mathcal J(V, \omega)$ with non-positive holomorphic bisectional curvature; moreover, its holomorphic sectional curvature is bounded above by $-2/n$.

\subsection{Geodesics}

\subsubsection{K\"ahler metric geodesics}

Let $(X, \omega)$ be a fixed $n$-dimensional compact K\"ahler manifold.  Consider the following Mabuchi space of K\"ahler potentials 
$$
\mathcal K:=\{\phi\in C^\infty (X, \mathbb R): \omega+i\partial\dbar \phi >0 \}
$$ 
on $X$. Fix $\phi_0, \phi_1$ in $\mathcal  K$, if there exists a smooth function $\phi$ on a neighborhood of the closure of  
$$
\X:=\mathbb H_{0,1} \times X, \ \ \ \mathbb H_{0,1}:= \{ \tau \in \mathbb C: 0<{\rm Re} \, \tau<1 \},
$$ 
such that $\phi(0,x)=\phi_0(x)$ , $\phi(1,x)=\phi_1(x)$, $\phi$ does not depend on the imaginary part of $\tau$ and 
$$
(\omega+i\partial\dbar\phi)^{n+1} \equiv 0  \ \text{on}\ \X, \ \ \ \phi(t, \cdot)\in \mathcal K,
$$
then we say that $\{\phi(t, \cdot)\}_{t\in [0,1]}$ is a smooth geodesic in $\mathcal K$ connecting $\phi_0, \phi_1$. Associated with a smooth geodesic, the following trivial fibration
$$
p: (\X, \omega+i\partial\dbar \phi) \to \mathbb H_{0,1}
$$
is a Monge-Amp\`ere fibration.

\subsubsection{Convex function geodesics} If $\phi$ is a smooth, strictly convex function on $\mathbb R^n$, then we know that its gradient map
$$
\nabla \phi : x\mapsto (\phi_{x_1}(x), \cdots, \phi_{x_n}(x)), \ \ \phi_{x_j}:=\partial\phi/\partial x_j,
$$
defines a diffeomorphism from $\mathbb R^n$ onto an open set
$$
A_\phi:= \nabla \phi(\mathbb R^n)
$$
in $\mathbb R^n$. Moreover, one can check that $A_\phi$ is convex in $\mb{R}^n$.

\begin{defn} Let $A$ be a bounded open convex set in $\mathbb R$.  A smooth, strictly convex function $\phi$ on $\mathbb R^n$ is said to be of type $A$ if $A_\phi=A$. We call denote by $\mathcal C_A$ the space of type $A$ functions.
\end{defn}

Note that  $\mathcal C_A$ is not empty. In fact, if $\psi$ is a smooth, strictly convex function on $A$ that tends to infinity at the boundary of $A$, then its Legendre transform
$$
\psi^*(x):=\sup_{y\in  A} x\cdot y-\psi(y),  \ \ \forall \ x\in \mathbb R,
$$
lies in $\mathcal C_A$. $A_{\phi+\psi}=A_\phi+A_\psi$ implies that $\mathcal C_A$ is a convex set.

 The Legendre transform of $\phi \in \mathcal C_A$ is defined by 
$$
\phi^*(y):= \sup_{x\in \mathbb R} x\cdot y-\phi(x), \ \ \forall \ y\in A.
$$
We know that $\phi^*$ is smooth and strictly convex  on $A$. Moreover, if $\phi_0, \phi_1 \in \mathcal C_A$, then
\begin{equation}\label{eq:c-geo}
\phi: (t,x) \mapsto (t\phi^*_1+(1-t)\phi^*_0)^*(x)
\end{equation}
satisfies
$$
MA (\phi)=0
$$
on $[0,1] \times \mathbb R^n$, where $MA(\phi)$ denotes the determinant of the full Hessian of $\phi$.

\begin{defn} We call $\phi$ defined in \eqref{eq:c-geo} the geodesic between $\phi_0, \phi_1 \in \mathcal C_A$.
\end{defn}

Let $\X:= [0,1] \times \mathbb R^n \times \mathbb R^{n+1}\subset \mathbb C^{n+1}$ be the natural complexification of $[0,1] \times \mathbb R^n$. Think of $\phi$ as a function on $\X$, then 
$$
p: (\X, i\partial\dbar \phi) \to \B, \ \ \B:=[0,1] \times \mathbb R \subset \mathbb C,
$$  
is a (non-proper) Monge-Amp\`ere fibration. 

\subsubsection{Hermitian form geodesics} Denote by 
$\mathcal H$ the space of Hermitian forms on $\mathbb C^n$. Let $\{e_j\}$ be the canonical basis of $\mathbb C^n$ then a Hermitian form, say $\omega\in \mathcal H$, can be written 
as
$$
\omega=i\sum_{j,k=1}^n a_{j\bar k} \, e_j^* \wedge \overline{e_k^*},
$$
where $A:=(a_{j\bar k})$ satisfies
$$
a_{j\bar k}=\overline{a_{k\bar j}}
$$
and $\sum a_{j\bar k} \xi^j \bar \xi^k >0$ if $\xi\neq 0$. Thus we can identify $\omega$ with a Hermitian matrix $A$. Now let 
$$
\mathbb A:=\{A_t\}_{t\in [0,1]}
$$ 
be a smooth family (smooth on a neighborhood of $[0,1]$)  of Hermitian matrices. We know that $\mathbb A$ defines a smooth metric on the trivial bundle 
$$
p:  \X\to \B, \ \  \X:=[0,1]\times \mathbb R \times \mathbb C^n, \ \  \B:=[0,1]\times \mathbb R\subset \mathbb C, 
$$
with Chern curvature 
$$
\Theta_{tt}(\mathbb A) e_j=\sum (a_{j\bar k, t} a^{\bar k l})_t e_l= \sum (a_{j\bar k, tt} a^{\bar k l} - a_{j\bar k, t} a_{p\bar q, t}a^{\bar k p} a^{\bar q l})e_l,
$$
where $(a^{\bar k l})$ denotes the inverse matrix of $(a_{j\bar k})$ and $f_{,t}$ denotes the derivative of $f$ with respect to $t$. Think of 
$$
\phi(t,z):= \sum a_{j\bar k}(t) z^j\bar z^k
$$
as a function on $\X$. Then $i\partial\dbar \phi$ defines a relative K\"ahler form on $\X$. A direct computation gives

\begin{prop}\label{pr5.3} $\Theta_{tt}(\mathbb A) \equiv 0$ if and only if $(i\partial\dbar \phi)^{n+1} \equiv 0$. 
\end{prop}

Now we know that if $\mathbb A$ is flat, then 
$$
p: (\X, i\partial\dbar \phi) \to \B
$$
is a (non-proper) Monge-Amp\`ere fibration.

\begin{defn} We say that $\mathbb A$ is the geodesic between $A_0$ and $A_1$ if $\Theta_{tt}(\mathbb A) \equiv 0$.
\end{defn}

\end{document}